\documentclass[11pt,leqno]{amsart}
\usepackage{amsmath,amsfonts,amssymb,amscd,amsthm,amsbsy,upref}
\usepackage[alphabetic]{amsrefs}
\usepackage{color}
 \usepackage{comment}
\usepackage{mathptmx} 
\usepackage{hyperref}
\usepackage{tikz}

\newtheorem{thm}{Theorem}[section]
\newtheorem{lem}[thm]{Lemma}
\newtheorem{cor}[thm]{Corollary}
\newtheorem{prop}[thm]{Proposition}

\newtheorem{prob}[thm]{Problem}

\theoremstyle{definition}
\newtheorem{defin}[thm]{Definition}

\newtheorem{rem}[thm]{Remark}

\newtheorem{remark}[thm]{Remark}

\newtheorem{thmAlfa}{Theorem}

\newtheorem{corAlfa}[thmAlfa]{Corollary}

\def\N{{\mathbb N}}
\def\M{{\mathbb M}}
\def\L{{\mathbb L}}
\def\R{{\mathbb R}}

\def\cA{{\mathcal A}}

\def\cE{{\mathcal E}}

\def\cM{{\mathcal M}}

\newcommand{\keq}{\!=\!}
 
\newcommand{\kleq}{\!\leq\!}
\newcommand{\kge}{\!\ge\!}
\newcommand{\kle}{\!<\!}
\newcommand{\kgr}{\!>\!}
\newcommand\kin{\!\in\!}
\newcommand{\ksubset}{\!\subset\!}

\newcommand{\kplus}{\!+\!}

\newcommand{\lb}{\bar l}
\newcommand{\nb}{\bar n}
\newcommand{\mb}{\bar m}

\newcommand{\eb}{\bar e}
\newcommand{\xb}{\bar x}

\def\vp{\varepsilon}

\newcommand{\xt}{{\tilde x}}

\newcommand{\cof}{\text{\rm cof}}
\newcommand{\spa}{\text{\rm span}}
\newcommand{\supp}{\text{\rm supp}}
\newcommand{\ie}{{\it i.e.,\ }}
\newcommand{\co}{\mathrm{c}_0}
\newcommand{\diam}{\text{\rm diam}}
\newcommand{\Lip}{\text{\rm Lip}}

\newcommand{\bin}{\ensuremath{\mathsf{B}}}
\newcommand{\ham}{\ensuremath{\mathsf{H}}}

\begin{document}

\allowdisplaybreaks


\title{the geometry of Hamming-type metrics and their embeddings into Banach spaces}
\author{F.~Baudier}
\address{F.~Baudier, Department of Mathematics, Texas A\&M University, College Station, TX 77843, USA}
\email{florent@math.tamu.edu}

\author{G.~Lancien}
\address{G.~Lancien, Laboratoire de Math\'ematiques de Besan\c con, Universit\'e Bourgogne Franche-Comt\'e, 16 route de Gray, 25030 Besan\c con C\'edex, Besan\c con, France}
\email{gilles.lancien@univ-fcomte.fr}

\author{P.~Motakis}
\address{P.~Motakis, Department of Mathematics, University of
Illinois at Urbana-Champaign, Urbana, IL 61801, U.S.A.}
\email{pmotakis@illinois.edu}

\author{Th.~Schlumprecht}
\address{Th.~Schlumprecht, Department of Mathematics, Texas A\&M University, College Station, TX 77843-3368, USA, and Faculty of Electrical Engineering,
Czech Technical University in Prague, Zikova 4, 16627, Prague, Czech Republic}
\email{schlump@math.tamu.edu}

\thanks{The first named author was supported by the National Science
Foundation under Grant Number DMS-1800322.
The second named author was supported by the French
``Investissements d'Avenir'' program, project ISITE-BFC (contract
 ANR-15-IDEX-03).
The third named author was  supported by the National Science Foundation
under Grant Numbers DMS-1600600 and DMS-1912897.
The fourth named author was supported by the National Science Foundation under Grant Numbers DMS-1464713 and DMS-1711076 .}

\maketitle
\begin{abstract}
Within the class of reflexive Banach spaces, we prove a metric characterization of the class of asymptotic-$\co$ spaces in terms of a bi-Lipschitz invariant which involves metrics that generalize the Hamming metric on $k$-subsets of $\N$. We apply this characterization to show that the class of separable, reflexive, and asymptotic-$\co$ Banach spaces is non-Borel co-analytic. Finally, we introduce a relaxation of the asymptotic-$\co$ property, called the asymptotic-subsequential-$\co$ property, which is a partial obstruction to the equi-coarse embeddability of the sequence of Hamming graphs. We present examples of spaces that are asymptotic-subsequential-$\co$. In particular $T^*(T^*)$ is asymptotic-subsequential-$\co$ where $T^*$ is Tsirelson's original space.
\end{abstract}

 \setcounter{tocdepth}{3}
\tableofcontents
\section{Introduction}\label{S:1}
A central theme of the Ribe Program is to find metric characterizations of linear properties of Banach spaces. We refer to \cite{Naor12}, \cite{Ball13}, and \cite{Naor18} for a discussion of the origins, motivations, applications, and the depth of the Ribe Program. There are various forms of metric characterizations, the most common ones being expressed in terms of Poincar\'e-type/concentration inequalities, or in terms of containment in a metric sense of a sequence of graph metrics. If a class $\mathcal{C}$ of separable Banach spaces coincides with the class of Banach spaces equi-coarsely (or equi-bi-Lipschitzly) containing some sequence $(M_k)_k$ of metric spaces, then $\mathcal{C}$ would be an analytic class since it can be shown that the latter class is analytic (in the Effros-Borel structure).

The following metric characterization, in terms of a concentration inequality, was proved in \cite{BLMS_JIMJ20} and was used to show that the class of reflexive and asymptotic-$\co$ Banach spaces is coarsely rigid.
\begin{thm}[\cite{BLMS_JIMJ20}]\label{thm}
A Banach space $X$ is reflexive and asymptotically $c_0$ if and only if there exists $C\geq 1$ such that for every $k\in\N$ and every Lipschitz map $f:\big([\N]^k,d^{(k)}_\ham\big)\to X$ there exists an infinite subset $\M$ of $\N$ so that
\begin{equation}\label{eq:concentration}
\sup_{\mb,\nb\in[\M]^k}\|f(\mb)-f(\nb)\| \leq C\mathrm{Lip}(f).
\end{equation}
\end{thm}

In Theorem \ref{thm}, $d^{(k)}_\ham$ is the Hamming metric on the set $[\N]^k$ of $k$-subsets of $\N$, and we will simply denote $\big([\N]^k,d^{(k)}_\ham\big)$ by $\ham_k^\omega$. The concentration inequality \eqref{eq:concentration} prevents the equi-coarse embeddability of the sequence of Hamming graphs $(\ham_k^\omega)_k$ into any reflexive and asymptotic-$c_0$ Banach space.  The converse does not hold since it was shown in \cite{BLMS_JIMJ20} that there are quasi-reflexive (and not reflexive) asymptotic-$\co$ Banach spaces that do not equi-coarsely contain $(\ham_k^\omega)_k$. Therefore the coarse (or Lipschitz) geometry of the Hamming graphs cannot be used directly to compute the descriptive set theoretic complexity of the class of separable, reflexive and asymptotic-$\co$ Banach spaces. It follows from \cite[Theorem3]{DodosFerenczi2008} that if this class were analytic, then there would exist a separable reflexive space containing isomorphic copies of all members of this class. However, in \cite[Remark on page 120]{OdellSchlumprechtZsak2008} it is observed that if a  separable space contains isomorphic copies of all reflexive and asymptotic-$c_0$ spaces then it must contain an isomorphic copy of $c_0$, barring it from being reflexive. In conclusion, the class of separable, reflexive and asymptotic-$\co$ Banach spaces is non-analytic and in particular non-Borel.

In this article, we continue our investigation of the metric geometry of the Hamming graphs and we introduce a useful class of metrics on $[\N]^k$ which generalizes the Hamming metric. These \emph{Hamming-type} metrics are generated by certain basic sequences of Banach spaces and, relying on geometric arguments, they  can be used to prove that the class of separable, reflexive and asymptotic-$\co$ Banach spaces is co-analytic.

\begin{defin}\label{D:5.1}
Let $\eb = (e_j)_{j\in\N}$ be a normalized $1$-suppression unconditional basis of a Banach space  $E$. For every $k\in\N$ we define $d^{(k)}_{\eb}:[\N]^k\times[\N]^k\to\R$ as follows: If $\mb=\{m_1,m_2,\ldots,m_k\}$, $\nb=\{n_1,n_2,\ldots,n_k\}$ are in $[\N]^k$
(both sets written in increasing order) and $F = \{j: m_j\neq n_j\}$ then $d^{(k)}_{\eb}(A,B) = \big\|\sum_{j\in F}e_j\big\|_E$.
\end{defin}

We will justify that $d^{(k)}_{\eb}$ is indeed a metric in Section \ref{sec:3.2}. The metric $d^{(k)}_{\eb}$ is dominated by the Hamming metric $d^{(k)}_\ham$ and coincides with it if $(e_j)_{j\in\N}$ is the canonical basis of $\ell_1$. Also, if $\bar e = (e_j)_j$ is not equivalent to the unit vector basis of $c_0$ then the sequence of metric spaces $\big([\N]^k,d^{(k)}_{\eb}\big)_k$ is hereditarily unbounded, in the following sense:
\[\lim_{k\to\infty}\inf_{\M\in[\N]^\omega}\diam\big([\M]^k,d_{\eb}^{(k)}\big) = \lim_k\big\|\sum_{i=1}^ke_i\big\| = \infty.\]
Recall that for two metric space $X$ and $Y$ are two metric spaces, the \emph{$Y$-distortion} of $X$, denoted $c_Y(X)$, is defined as the infimum of those $D\in[1,\infty)$ such that there exist $s\in(0,\infty)$ and a map $f\colon X\to Y$ so that for all $x,y\in X$
  \begin{equation}\label{E:distortion}
    s\cdot d_{X}(x,y)\leq d_{Y}\big(f(x),f(y)\big)\leq
    s\cdot D\cdot d_{X}(x,y).
  \end{equation}
When \eqref{E:distortion} holds we say that $X$ bi-Lipschitzly embeds into $Y$ with distortion at most $D$. Within the class of separable reflexive Banach spaces, we prove the following metric characterization of the class of asymptotic-$\co$ Banach spaces.

\begin{thmAlfa}\label{T:A} Let $X$ be a separable reflexive Banach space.

\noindent $X$ is asymptotic-$c_0$ if and only if for all $1$-suppression unconditional
sequence $\bar e = (e_j)_j$ such that
$\lim_k\inf_{\M\in[\N]^\omega}\diam\big([\M]^k,d_{\eb}^{(k)}\big)=\infty$ one has $\sup_{k\in\N}c_X\big([\N]^k,d^{(k)}_{\eb}\big)=\infty$.
\end{thmAlfa}

Theorem \ref{T:A} which is the main result of Section \ref{sec:3}, cannot be drawn from the statement of Theorem \ref{thm} alone. The key difference is that it provides the existence of an embedding for a ``Hamming-type'' metric instead of the non existence of a concentration phenomenon. As in \cite{BLMS_JIMJ20}, the proof of Theorem \ref{T:A} relies in large part on a theorem of Freeman, Odell, Sari, and Zheng \cite{FOSZ2017} which establishes a deep and unexpected relation between the asymptotic structure of a Banach space and its asymptotic models. However to obtain the finer geometric information in Theorem \ref{T:A}, another ingredient is required. A crucial unconditionality property for normalized weakly null arrays of finite height is proved using an asymptotic notion of a third kind
namely  joint spreading models introduced in \cite{AGLM2017}. The following complexity result follows from Theorem \ref{T:A} and an application of the Souslin operation from descriptive set theory.

\begin{corAlfa}\label{C:B}
The class of separable, reflexive, and asymptotic-$\co$ Banach spaces is non-Borel co-analytic.
\end{corAlfa}

The quantity $\sup_{k\in\N}c_X\big([\N]^k,d^{(k)}_{\ham}\big)=\infty$ cannot be a substitute for the metric invariant in Theorem \ref{T:A} since it follows from \cite{KaltonRandrianarivony2008} that $\sup_{k\in\N}c_{\ell_2}\big([\N]^k,d^{(k)}_{\ham}\big)=\infty$, and the Hilbert space $\ell_2$ is not asymptotic-$\co$. Identifying the class of Banach spaces which equi-bi-Lipschiztly, or equi-coarsely, contain the Hamming graphs is a central problem in nonlinear geometry of Banach spaces. The goal of Section \ref{sec:4} is to provide new insights on this problem. With the previous knowledge on the geometry of the Hamming graphs, there still existed a possibility that the metric invariant in Theorem \ref{T:A} could be substituted with the failure of equi-coarse embeddability of the Hamming graphs. We examine this possibility in Section \ref{sec:4}. We already know from \cite{BaudierLancienSchlumprecht2018} that a Banach space admitting an unconditional spreading model not equivalent to the unit vector basis of $c_0$ equi-coarsely contain the Hamming graphs. We must therefore draw our attention to non-asymptotic-$c_0$ Banach spaces all of their spreading models are isomorphic to $c_0$. A particularly interesting example, the space $T^*(T^*)$, is studied to a great extent in Section \ref{sec:4.2}. We introduce a new linear property, which we called {\em asymptotic-subsequential-$c_0$}, that is strong enough to rule out the existence of a sequence of equi-coarse embeddings of the Hamming graphs of certain canonical types.
\begin{defin}
Let $X$ be an infinite dimensional Banach space. We say that $X$ is an {\em asymptotic-subsequential-$\co$} space if there exists a constant $C\geq 1$ so that for all $n\in\mathbb{N}$ there exists an $N\in\mathbb{N}$ satisfying the following: whenever $E=(\R^N,\|\cdot\|_E)$
  is in the
 $N$-th asymptotic structure of $X$  (to be defined in Subsection \ref{subsec:2.4}) then there are $i_1<\cdots<i_n$ so that $(e_{i_k})_{k=1}^n$ is $C$-equivalent to the unit vector basis of $\ell_\infty^n$, where $(e_j)_{j=1}^N$ is the unit basis in $\R^N$.
\end{defin}

We then show that a $T^*$-sum of countable copies of $T^*$ spaces is an asymptotic-subsequential-$c_0$ space,
but not necessarily asymptotic $c_0$.

\begin{thmAlfa}\label{T:C}
The space $T^*(T^*)$ is asymptotic-subsequential-$c_0$ but not asymptotic-$c_0$.
\end{thmAlfa}

\section{Preliminaries}
\label{sec:2}

\subsection{Coarse and Lipschitz embeddings}
\label{sec:2.1}
We introduce some convenient terminology and notation that will allow us to treat all at once various embedding notions.

\begin{defin}
Let $X$ and $Y$ be metric spaces. Let $\rho,\omega\colon [0,\infty)\to[0,\infty)$. We say that $X$ $(\rho,\omega)$-embeds into $Y$ if there exists $f\colon X\to Y$ such that for all $x,y\in X$ we have
\begin{equation}\rho(d_X(x,y))\le d_Y(f(x),f(y))\le \omega(d_X(x,y)).
\end{equation}

If $\{X_i\}_{i\in I}$ is a collection of metric spaces. We say that $\{X_i\}_{i\in I}$ $(\rho,\omega)$-embeds into $Y$ if for every $i\in I$, $X_i$ $(\rho,\omega)$-embeds into $Y$.
\end{defin}

We will say that $\{X_i\}_{i\in I}$ equi-coarsely embeds into $Y$ if there exist non-decreasing functions $\rho,\omega\colon [0,\infty)\to[0,\infty)$ such that $\lim_{t\to\infty}\rho(t)=\infty$ and $\{X_i\}_{i\in I}$ $(\rho,\omega)$-embeds into $Y$. With an abuse of notation we say that $\{X_i\}_{i\in I}$ equi-bi-Lipschiztly embeds into $Y$ if there exist $s,D>0$ such that $\{X_i\}_{i\in I}$ $(\rho,\omega)$-embeds into $Y$, with $\rho(t)=st$ and $\omega(t)=sDt$. Note that equi-bi-Lipschitz embeddability is a stronger condition than merely assuming that $\sup_{i\in I}c_Y(X_i)<\infty$ since it does not allow for arbitrarily large or arbitrarily small scaling factors in \eqref{E:distortion}.
However if $Y$ is a Banach space, rescaling is possible and the two notions coincide.

\subsection{Projective hierarchy and the Souslin operation}
\label{sec:2.2}
Let us recall a few basics from descriptive set theory. We refer the reader to the book by Kechris \cite{Kechris1995}, where all the proofs and details can be found.
 A measurable space $(X,\cM)$ is said to be a \emph{standard Borel space} if there exists a Polish topology $\tau$ (\ie separable and completely metrizable) on $X$ such that the Borel $\sigma$-algebra generated by $\tau$ coincides with the $\sigma$-algebra $\cM$.  A typical standard Borel space is Baire space, \ie $\N^\omega$ equipped with the Borel $\sigma$-algebra generated by the product of the discrete topology. The set of all \emph{closed subsets} of the Banach space $C[0,1]$, which is a Polish space, is a standard Borel space when equipped with the Effros-Borel structure. Invoking a selection theorem of Kuratowski and Ryll-Nardzewski together with the classical result that every separable Banach space isometrically embeds into $C[0,1]$, the class
  \[\mathsf{SB}:=\{X\colon X\text{ is a separable Banach space}\}\]
  can be considered as a standard Borel space. With this identification in mind, classes of separable Banach spaces become subsets of $\mathsf{SB}$, and the topological complexity results in this paper will always refer to this standard Borel structure. We are interested in the projective complexity. The projective hierarchy is built using the operations of projection (or equivalently of continuous image) and complementation. The $0$-level of the projective hierarchy consists of the Borel sets. The next level comprises \emph{analytic sets }which are exactly the continuous images of Borel sets, and \emph{co-analytic sets} which are the complements of analytic sets. We will not need to discuss higher levels of the projective hierarchy which can be obtained by iterating the projection and complementation operations.
An immediate corollary of Souslin first separation theorem establishes a fundamental connection between the Borel hierarchy and the projective hierarchy. More precisely, Borel sets are exactly those sets that are analytic and co-analytic. The analytic and co-analytic classes can be seen to be stable under countable intersection or countable unions. A fact of crucial importance to us is that the class of analytic sets is also stable under the Souslin operation. Let $\N^{\omega}$ be the set of all sequences of natural numbers. If $x=(x_1,x_2,\dots)\in\N^\omega$ and $k\in \N$ we write $x_{\restriction k}:=(x_1,x_2,\dots,x_k)$ the restriction of $x$ to its first $k$ terms. The Souslin operation, denoted $\cA$ in honor of Alexandrov, applies to a collection of sets $\{A_s\colon s\in \N^{<\omega}\}$ where $\N^{<\omega}$ denotes all the \emph{finite} sequences of natural numbers, and is defined as
\begin{equation}
\cA(\{A_s\}_{s\in \N^{<\omega}}):=\bigcup_{t\in\N^\omega}\bigcap_{k\in\N}A_{t_{\restriction k}}.
\end{equation}

It is easy to observe that the Souslin operation subsumes countable union and countable intersection. Moreover, the Souslin operation is idempotent and every analytic set can be obtained via an application of the Souslin operation over a collection of closed sets. We refer to \cite[Chapter III, Section 25]{Kechris1995} for properties of Souslin schemes. The following proposition will be needed in Section \ref{sec:3.1}:

\begin{prop}
Let $\{A_s\colon s\in \N^{<\omega}\}$ be a collection of analytic sets, then $\cA(\{A_s\}_{s\in \N^{<\omega}})$ is analytic.
\end{prop}

\subsection{Asymptotic models and spreading models of unconditional sums}
\label{sec:2.3}

In this section we recall the definitions of spreading and asymptotic models and prove two results about the spreading models of complemented sums. These results will be used in Section \ref{sec:4}.
For two basic sequences $(x_i)$ and $(y_i)$ in some Banach spaces $X$ and $Y,$ respectively, and $C\ge 1$, we say that   $(x_i)$ and $(y_i)$ are {\em $C$-equivalent}, and we write $(x_i)\sim_C(y_i)$, if there are positive  numbers $A$ and $B$,
with $C=A\cdot B$,
so that for all $(a_j)\in c_{00}$, the vector space of all sequences  $x=(\xi_j)$ in $\R$ for which the support $\supp(x)=\{j: \xi_j\not=0\}$ is finite,  we have
$$\frac1{A}\Big\|\sum_{i=1}^\infty a_i x_i\Big\|_X\le \Big\|\sum_{i=1}^\infty a_i y_i\Big\|_Y\le B\Big\|\sum_{i=1}^\infty a_i x_i\Big\|_X.$$
In that case we say that  $\frac1A$ is the {\em lower estimate} and $B$ the {\em upper estimate of $(y_i)$ with respect to $(x_i)$}.
Note that $(x_i)$ and $(y_i)$ are  $C$-equivalent if and only $C\ge \|T\|\cdot\|T^{-1}\|$, where the linear operator  $T:   \spa(x_i:i\in\N) \to  \spa(y_i:i\in\N)$, is defined by $T(x_i)=y_i$, $i\kin\N$.

If $(e_i)$ is a Schauder basis of a Banach space $X$, we recall that $(x_n)$ is a {\em block sequence} in $X$ with respect to the basis $(e_i)$ if for all $n\in \N$:
\[x_n\neq 0\ \ \text{and}\ \ \max( \supp(x_n))<\min(\supp(x_{n+1})).\]

For a sequence of Banach spaces $(X_k)_{k\in\N}$, and a Banach space $U$, which has a $1$-unconditional basis $(u_j)$ we denote the {\em $U$-sum of the $X_k$'s}, by $\big(\oplus_{k=1}^\infty  X_k\big)_U$. This is the space of all sequences $\xb=(x_k)$, with $x_k\in X_k$, for $k\in\N$, such that the series $\sum_{k=1}^\infty  \|x_k\| u_k$ converges in $U$, and equipped with the norm
$$\|\xb\|= \Big\|\sum_{k=1}^\infty  \|x_k\| u_k\Big\|_U.$$
If all the $X_k$'s are the same space $X$ we also write $U(X)$ instead of $\big(\oplus_{k=1}^\infty X\big)_U$.

Our first proposition is about spreading models of unconditional sums of Banach spaces. Spreading models were introduced by Brunel and Sucheston in \cite{BrunelSucheston1974}. We recall the definition.
Let $E$ be a Banach space with a normalized basis $(e_i)$ and let
$(x_i)$ be a basic sequence in a Banach space $X$. We say that $E$ with its basis $(e_i)$ is a {\em spreading model of $(x_i)$},  if  there is a   null-sequence $(\vp_n)\subset(0,1)$, so that for  all $n$, all $(a_i)_{i=1}^n \subset[-1,1]$ and $n\le k_1<k_2<\ldots<k_n$, it follows that
 \begin{equation*}
 \Bigg|\Big\|\sum_{i=1}^n a_i x_{k_i}\Big\|_X  - \Big\| \sum_{i=1}^n a_i e_i \Big\|_E\Bigg|<\vp_n
 \end{equation*}
or, in other words, if
$$\lim_{k_1\to\infty}\lim_{k_2\to\infty} \ldots \lim_{k_n\to\infty} \Big\|\sum_{j=1}^n a_j x_{k_j}\Big\|_X=\Big\|\sum_{j=1}^n a_j e_j\Big\|_E.$$

Using Ramsey's Theorem it can be shown that every normalized basic sequence has a subsequence which admits a spreading model.

\begin{prop}\label{P:3.12}
Let $1\leq p \leq\infty$, $A,B,C,D\geq 1$, and $(X_n)$ be a sequence of Banach spaces so that for all $n\in\mathbb{N}$ any spreading model generated by a normalized weakly null sequence in $X_n$ is equivalent to the unit vector basis of $\ell_p$  (or $c_0$ if $p=\infty$) with $\frac1C$-lower and $D$-upper estimates. Let also $U$ be a reflexive Banach space with a $1$-unconditional  basis $(u_n)$ satisfying the following  property:
\begin{enumerate}
\item[($*$)]
For every finitely supported $x_0\in S_U$, every normalized block sequence $(x_n)_n$ in $U$, and every $k\in\mathbb{N}$ there exist $n_1<\cdots <n_k$ so that the sequence $(x_0,x_{n_1},\ldots,x_{n_k})$ is equivalent to the unit vector basis of $\ell_p^{k+1}$ with $\frac1A$-lower and $B$-upper estimates.
\end{enumerate}

 Then every spreading model generated by a weakly null normalized sequence in $(\oplus_{n=1}^\infty X_n)_U$ is equivalent to the unit vector basis of $\ell_p$ with $\frac1{ABC}$-lower and $ABD$-upper estimates.
\end{prop}

\begin{proof}[Proof of Proposition \ref{P:3.12}] We assume that $p<\infty$, for $p=\infty$ the proof is similar.
 Assume that $x^{(m)}=\sum_{j=1}^\infty x^{(m)}_j\kin \big(\oplus_{j=1}^\infty X_j\big)_U$,  for $m\kin\N$, with $x^{(m)}_j\in X_j$, for  $j\kin\N$, and $\|x^{(m)}\|=\Big\|\sum_{j=1}^\infty \|x^{(m)}_j\| u_j\Big\|=1,$ and assume that $(x^{(m)})_{m=1}^\infty$ converges weakly to $0$. It is enough to show that for fixed $k\in\N$, $(a_i)_{i=1}^k$ in $S_{\ell_p^k}$ and $\vp>0$, there is a subsequence $(\xt^{(m)})_m$ of $(x^{(m)})_m$ so that for all $m_1<m_2<\ldots <m_k$
\begin{equation}\label{E:3.12.1}
\frac{1-\vp}{ABC} \le \Big\|\sum_{i=1}^k a_i \xt^{(m_i)} \Big\|\le  ABD(1+\vp).
\end{equation}
Then a straightforward diagonalization argument will prove our claim. We define $z_m=\sum_{j=1}^\infty \|x^{(m)}_j \|u_j$, for $m\in\N$. Since $U$ is reflexive we can assume, after passing to a  subsequence, that $z_m$ is weakly converging to some  $z=\sum_{j=1}^\infty  b_j u_j$.  Since we need to show \eqref{E:3.12.1} for a fixed $k$ and a fixed $(a_i)_{i=1}^k\in S_{\ell_p^k}$, we can assume, after passing again to a subsequence  and to arbitrarily small perturbations, that $z=\sum_{j=1}^{l_0} b_j u_j$ for  some  $l_0\in\N$, and  that there are  intervals $I_m\subset \N$, with $l_0<\min(I_1)\le \max(I_1)<\min(I_2)\le\max(I_2)<\ldots$, so that  for all $m\in\N$ we can write $z_m$ as
\begin{equation} \label{E:3.12.3}  z_m=\sum_{j=1}^{l_0} b_j u_j+ \sum_{j\in I_m}\|x^{(m)}_j\| u_j, \text{ and } b_j=\|x^{(m)}_j\|,   \text{ for $j=1,2,\ldots, l_0$.}\end{equation}
and thus
 \begin{align}\label{E:3.12.2}&x^{(m)}=\sum_{j=1}^{l_0} x^{(m)}_j +\sum_{j\in I_m}x^{(m)}_j.
                 \end{align}
By the  assumption on $X_j$, $j\in\N$ and because the sequences $(x_j^{(m)})_j$ are weakly null, we also can assume, after passing to a subsequence  that for $1\le m_1<m_2<\ldots<m_k$ and every $j=1,2,\ldots, l_0$ we have
   \begin{align}\label{E:3.12.4}\frac{1-\vp}C \Bigg( \sum_{i=1}^k |a_i|^p \|x^{(m_i)}_j\|^p\Bigg)^{1/p}&
   \le \Big\|\sum_{i=1}^k a_i x^{(m_i)}_j\Big\|
    \le (1+\vp)D\Bigg( \sum_{i=1}^k |a_i|^p \|x^{(m_i)}_j\|^p\Bigg)^{1/p}.\end{align}
Finally, letting  $y_0=z\in  \spa(u_j:j=1,2,\ldots, l_0)$ and $y_m=\sum_{j\in I_m}\|x^{(m)}_j \|u_j$, for $m\in\N$, we can use the property ($*$) of $U$, and,  again after passing to a subsequence, assume that for all $m_1<m_2<\ldots<m_k$
  \begin{align}\label{E:3.12.5}  \frac{1-\vp}{A}\Bigg(\|y_0\|^p+\sum_{i=1}^k  |a_i|^p \|y_{m_i}\|^p\Bigg)^{1/p}
  &\le \Big\|y_0+\sum_{i=1}^k  a_i y_{m_i}\Big\|\\
  &\le (1+\vp)B\Bigg(\|y_0\|^p\kplus\sum_{i=1}^k  |a_i|^p \|y_{m_i}\|^p\Bigg)^{1/p}\!\!.\notag\end{align}
Using \eqref{E:3.12.3} and \eqref{E:3.12.4} we deduce from the $1$-unconditionality of $(u_j)$ that
 \begin{align}\label{E:3.12.6}
 \Big\|\sum_{j=1}^{l_0}  \Big\|\sum_{i=1}^k a_i x^{(m_i)}_j\Big\|u_j   \Big\|&\le (1\kplus\vp)D \Big\|\sum_{j=1}^{l_0}   \Big( \sum_{i=1}^k |a_i|^p b_j^p\Big)^{1/p}u_j   \Big\| \\
 &= (1\kplus\vp)D \Big( \sum_{i=1}^k |a_i|^p\Big)^{1/p}\Big\| \sum_{j=1}^{l_0} b_j u_j\Big\| \notag \\
 &= (1\kplus\vp)D \Big\|\sum_{j=1}^{l_0}  \|x^{(m_i)}_j\| u_j\Big\|.\notag
 \end{align}

We therefore deduce that
\begin{align*}
 \Big\|&\sum_{i=1}^k a_i x^{(m_i)} \Big\|\\
 &=\Bigg\|\sum_{j=1}^\infty  \Big\|\sum_{i=1}^k a_i x^{(m_i)}_j\Big\|u_j\Bigg\|\\
 &=\Bigg\|\sum_{j=1}^{l_0}  \Big\|\sum_{i=1}^k a_i x^{(m_i)}_j\Big\|u_j
 +\sum_{i=1}^k a_i\sum_{j\in I_{m_i}}   \|x^{(m_i)}_j \|u_j  \Bigg\|\\
  &\le (1+\vp) B\Bigg( \Big\|\sum_{j=1}^{l_0}  \Big\|\sum_{i=1}^k a_i x^{(m_i)}_j\Big\|u_j   \Big\|^p+
  \sum_{i=1}^k |a_i|^p \Big\|\sum_{j\in I_{m_i}}  \|x^{(m_i)}_j \|u_j  \Big\|^p  \Bigg)^{1/p} \text{ (By \eqref{E:3.12.5})}  \\
 &\le (1+\vp)B\Bigg( (1\kplus\vp)^pD^p \Big\|\sum_{j=1}^{l_0}  \|x^{(m_i)}_j\| u_j\Big\|^p
  +\sum_{i=1}^k|a_i|^p\Big\|\sum_{j\in I_{m_i}}  \|x^{(m_i)}_j \|u_j  \Big\|^p\Bigg)^{1/p}
\text{ (By \eqref{E:3.12.6})}\\
&\le (1\kplus\vp)^2BD\Bigg( \Big\|\sum_{i=1}^k |a_i|^p \sum_{j=1}^{l_0}  b_j u_j\Big\|^p
  +\sum_{i=1}^k|a_i|^p\Big\|\sum_{j\in I_{m_i}}  \|x^{(m_i)}_j\| u_j  \Big\|^p  \Bigg)^{1/p}\\
 &\qquad \Big(\text{Since $\|x^{m_i}_j \|=b_j$, for $j=1,2,\ldots l_0$, and  $i=1,2\ldots k$, and  $\sum_{i=1}^k |a_i|^p=1$}\Big)
 \\
 &= (1+\vp)^2 BD\Bigg(\sum_{i=1}^k |a_i|^p\Big(\Big\| \sum_{j=1}^{l_0}  b_j  u_j\Big\|^p
  +     \Big\|\sum_{j\in I_{m_i}}  \|x^{(m_i)}_j \|u_j  \Big\|^p    \Big) \Bigg)^{1/p} \\
&\le \frac{(1+\vp)^2}{1-\vp}BDA\Bigg(\sum_{i=1}^k |a_i|^p \Bigg\| \sum_{j=1}^{l_0}b_ju_j  + \sum_{j\in I_{m_i}} \|x^{(m_i)}_j\| u_j\Bigg\|^p\Bigg)^{1/p}     \text{ (By \eqref{E:3.12.5})} \\
&= \frac{(1+\vp)^2}{1-\vp}ABD\Bigg(\sum_{i=1}^k |a_i|^p \|x^{(m_i)}\|^p\Bigg)^{1/p}=\frac{(1+\vp)^2}{1-\vp}ABD.
\end{align*}
Similarly we show that
\begin{align*}
\Big\|&\sum_{i=1}^k a_i x^{(m_i)} \Big\|\ge \frac{(1-\vp)^2}{1+\vp}\frac1{ABC}.
\end{align*}
We deduce therefore \eqref{E:3.12.1} after readjusting $\vp$.
\end{proof}

The next proposition is about asymptotic models of unconditional sums of Banach spaces. Asymptotic models, which are a generalization of spreading models, were introduced by Halbeisen and Odell in \cite{HalbeisenOdell2004}, and is based on the behavior of infinite arrays (as opposed to a single array for spreading models). An {\em array of infinite height} in a Banach space $X$ is a family $\big(x^{(i)}_j:i, j\in \N\big)\subset X$. For an array $\big(x^{(i)}_j:i, j\in \N\big)$, we call the sequence  $(x^{(i)}_j)_{j\in\N}$  {\em the $i$-th row of the array}. We call an array weakly null  if all rows are weakly null. A {\em subarray} of an infinite array $\big(x^{(i)}_j:i\in \N, j\in \N\big)\subset X,$ is an array of the form $\big(x^{(i)}_{j_s}:i\in \N, s\in \N\big)$, where $(j_s)\subset \N$ is a subsequence.  Thus, for a subarray we are taking  the same subsequence in each row.

A basic sequence $(e_i)$  is called an {\em asymptotic model} of a Banach space $X$, if there exist
 an    infinite  array  $\big(x^{(i)}_j:i ,j\kin \N\big)\subset S_X$ and a   null-sequence $(\vp_n)\subset(0,1)$, so that for  all $n$, all $(a_i)_{i=1}^n \subset[-1,1]$ and $n\le k_1<k_2<\ldots  <k_n$, it follows that
 \begin{equation*}
 \Bigg|\Big\|\sum_{i=1}^n a_i x^{(i)}_{k_i}\Big\|  - \Big\| \sum_{i=1}^n a_i e_i \Big\|\Bigg|<\vp_n.
 \end{equation*}

One may think of spreading models as asymptotic models for arrays with identical rows, and thus part of the theory of asymptotic models is reminiscent of the spreading model theory of Brunel and Sucheston. For instance, in \cite{HalbeisenOdell2004} it was shown that an asymptotic model generated by a normalized weakly null array is $1$-suppression unconditional.

  \begin{prop}\label{P:3.5}\cite{HalbeisenOdell2004}*{Proposition 4.1 and Remark 4.7.5}
Assume that  $\big(x^{(i)}_j:i, j\in \N\big)\subset S_X$ is an infinite array, all of whose rows  are normalized and weakly null. Then there is  a subarray of $\big(x^{(i)}_j:i, j\in \N\big)$ which has a $1$-suppression unconditional  asymptotic model $(e_i)$.
\end{prop}

 \begin{prop}\label{P:3.13}
Let $1\leq p< \infty$, $C, D\geq 1$ and $(X_n)_n$ be a sequence of Banach spaces so that for any $n\in\mathbb{N}$ every asymptotic model generated by a  normalized weakly null  array in $X_n$ is equivalent to the unit vector basis of $\ell_p$ with $C$-lower and $D$-upper estimates. Then every asymptotic model generated by a weakly null normalized array in the space $(\oplus_{n=1}^\infty X_n)_p$ is equivalent to the unit vector basis of $\ell_p$ with $C$-lower and $D$-upper estimates.
\end{prop}

\begin{proof}  For $\M\subset \N$
we denote the canonical projection from  $\big(\oplus_{k=1}^\infty X_k\big)_{\ell_p}$ onto $\big(\oplus_{k\in \M}X_k\big)_{\ell_p}$
by $P_\M$ and we abbreviate $W=\big(\oplus_{k=1}^\infty X_k\big)_{\ell_p}$.

Let $(w_j^{(i)}: i,j\in\mathbb{N})$ be a normalized weakly null array in $W$. By passing to a subarray, i.e. by taking a common infinite  set $\L$ of $j$'s and relabeling the array $(w_j^{(i)}: i\in\mathbb{N},j\in \L)$, we may assume that it generates an asymptotic model $(e_i)_i$. Fix $m\in\mathbb{N}$ and scalars $a_1,\ldots,a_m$. Without loss of generality we may assume that $ (\sum_{i=1}^m|a_i|^p)^{1/p} = 1$. The goal is to show that
\[\frac{1}{C} \leq \lim_{j_1\to\infty} \lim_{j_2\to\infty }\cdots\lim_{j_m\to \infty}\Big\|\sum_{i=1}^ma_iw^{(i)}_{j_i}\Big\| \leq D.\]
In particular, we are only interested in the first $m$ sequences of the given array thus we may disregard the remaining ones. By passing to a further subarray we may assume that the scalars $\mu_n^{(i)} = \lim_j\|P_{\{n\}}w_j^{(i)}\|$, $1\leq i\leq m$, $n\in\mathbb{N}$ exist. Observe that for $i=1,\ldots,m$  we have by Fatou's Lemma:
\[\sum_{n=1}^\infty(\mu_n^{(i)})^p \leq \liminf_{j\to\infty}\|w_j^{(i)}\|^p = 1.\]
We fix $\delta>0$  to be small enough, so that for all $0\le a\le 1$ and all $0\le x\le m\delta^{1/p}$ we have that  $(a+x)^p \leq a^p + 2px$
and $|a-x|^p\ge a^p-p x$.

Choose $n_0\in\mathbb{N}$ appropriately large so that for $i=1,\ldots,m$ we
have
\begin{equation} \label{E:3.13.1} \sum_{n>n_0}(\mu_n^{(i)})^p <
\delta.
\end{equation}
We now pick an increasing sequence $(n_j)_j$ in $\N$ such that
 for each $j\in\mathbb{N}$ and $1\leq i\leq m$
\begin{equation} \label{E:3.13.2}
\big\|P_{(n_j,\infty)}w_j^{(i)}\big\|^p < \delta.
\end{equation}
By the definition of the scalars $\mu_n^{(i)}$, $n\in\mathbb{N}$, $1\leq i\leq
m$, and \eqref{E:3.13.1}, we can now pass to a new common subarray so that
the following condition  is satisfied
\begin{equation} \label{E:3.13.3}
\big\|P_{(n_0,n_{j'}]}w_j^{(i)}\big\|^p < \delta\text{ for any }j'<j
\text{ in $\N$  and any } 1\leq i\leq m.
\end{equation}
We put  $j_0 = 0$. We  calculate for any choice of $j_1<j_2<\cdots<j_m$
\begin{align}\label{E:3.13.4}
\big\|P_{(n_0,\infty)}\sum_{i=1}^ma_iw_{j_i}^{(i)}\big\|^p &=
\Big\|\sum_{i=1}^m \big(P_{(n_{j_{i-1}},n_{j_i}]} +
P_{(n_0,n_{j_{i-1}}]} +
P_{(n_{j_i},\infty]}\big)a_iw_{j_i}^{(i)} \Big\|\\ &\leq
\Big(\Big(\sum_{i=1}^m\big\|P_{(n_{j_{i-1}},n_{j_i}]}a_iw_{j_i}^{(i)}\big\|^p\Big)^{1/p}
+ m\delta^{1/p} + m\delta^{1/p}\Big)^p\notag\\ &\text{(by \eqref{E:3.13.2}
and \eqref{E:3.13.3}) }\notag\\ & \leq
\sum_{i=1}^m\big\|P_{(n_{j_{i-1}},n_{j_i}]}a_iw_{j_i}^{(i)}\big\|^p
+ 4mp\delta^{1/p}.
\notag \end{align}
A similar argument  (using the choice
of $\delta$) also yields that
\begin{equation*}
\big\|P_{(n_0,\infty)}\sum_{i=1}^ma_iw_{j_i}^{(i)}\big\|^p \geq
\sum_{i=1}^m\big\|P_{(n_{j_{i-1}},n_{j_i}]}a_iw_{j_i}^{(i)}\big\|^p
- 2mp\delta^{1/p}.
\end{equation*}
We slightly refine this calculation:
\begin{align}
\label{E:3.13.5}
\big\|P_{(n_0,+\infty)}&\sum_{i=1}^ma_iw_{j_i}^{(i)}\big\|^p\\
 &\geq \sum_{i=1}^m\Big| \big\|P_{(n_0,+\infty)}a_iw_{j_i}^{(i)}\big\| - \big\|\big(P_{(n_{j_i},\infty)}+P_{(n_0,n_{j_{i-1}}]}\big)a_iw_{j_i}^{(i)}\big\| \Big|^p - mp\delta^{1/p}\notag\\
 &\geq  \sum_{i=1}^m\big\|P_{(n_0,+\infty)}a_iw_{j_i}^{(i)}\big\|^p - 4mp\delta^{1/\delta}.\notag
\end{align}

We now wish to evaluate the norm of an initial segment. For
$n\keq1,\ldots,n_0$ define $F_n \keq \{1\leq i\leq m: \mu_n^{(i)}\neq 0\}$.
By our assumptions, we may assume that for $n\keq1,\ldots,n_0$ the array
$$(z_j^{n,(i)}: i\kin F_n, j\kin\mathbb{N}) =
\Bigg(\frac{P_{\{n\}}w^{(i)}_j}{\|P_{\{n\}}w^{(i)}_j\|}: i\in F_n,
j\kin\mathbb{N}\Bigg)$$ generates an asymptotic model that is equivalent to
the unit vector basis of $\ell_p$ with $C$-lower and $D$-upper estimates.
We now calculate an initial segment of the norm.
\begin{align*} \lim_{
j_1\to\infty}\ldots\lim_{j_m\to\infty}\Big\|P_{[1,n_0]}\Big(\sum_{i=1}^ma_iw_{j_i}^{(i)}\Big)\Big\|^p
&= \lim_{
j_1\to\infty}\ldots\lim_{j_m\to\infty}\sum_{n=1}^{n_0}\Big\|\sum_{i\in
F_n}a_i\|P_{\{n\}}w_{j_i}^{(i)}\|z_{j_i}^{n,(i)}\Big\|^p\\ &\leq
\sum_{n=1}^{n_0}D^p\sum_{i\in F_n}|a_i|^p(\mu_n^{(i)})^p\\ &=
D^p\sum_{n=1}^{n_0}\sum_{i=1}^m|a_i|^p(\mu_n^{(i)})^p\\ &=
D^p\sum_{i=1}^m|a_i|^p\lim_{
j_1\to\infty}\ldots\lim_{j_m\to\infty}\sum_{n=1}^{n_0}\|P_{\{n\}}w_{j_i}^{(i)}\|^p\\
&= \lim_{
j_1\to\infty}\ldots\lim_{j_m\to\infty}D^p\sum_{i=1}^m\Big\|P_{[1,n_0]}a_iw_{j_i}^{(i)}\Big\|^p.
\end{align*}
We deduce that for any $j_1<\cdots<j_m$ that are chosen
sufficiently large we have
\begin{equation} \label{E:3.13.6}
\Big\|P_{[1,n_0]}\Big(\sum_{i=1}^ma_iw_{j_i}^{(i)}\Big)\Big\|^p \leq
D^p\sum_{i=1}^m\Big\|P_{[1,n_0]}a_iw_{j_i}^{(i)}\Big\|^p + \delta^{1/p}.
\end{equation}
A similar argument yields that for  $j_1<\cdots<j_m$ that
are chosen sufficiently large we have
\begin{equation} \label{E:3.13.7}
\Big\|P_{[1,n_0]}\Big(\sum_{i=1}^ma_iw_{j_i}^{(i)}\Big)\Big\|^p \geq
\frac{1}{C^p}\sum_{i=1}^m\Big\|P_{[1,n_0]}a_iw_{j_i}^{(i)}\Big\|^p -
\delta^{1/p}.
\end{equation}
We can finally estimate the desired norm. For
$j_1<\cdots<j_m$ large enough and $\delta$ sufficiently small, by
\eqref{E:3.13.4} and \eqref{E:3.13.6},  we have
\begin{align*}
\Big\|\sum_{i=1}^ma_iw_{j_i}^{(i)}\Big\|^p& =
\Big\|P_{[1,n_0]}\Big(\sum_{i=1}^ma_iw_{j_i}^{(i)}\Big)\Big\|^p  +
\Big\|P_{(n_0,\infty)}\Big(\sum_{i=1}^ma_iw_{j_i}^{(i)}\Big)\Big\|^p\\
&\leq
\sum_{i=1}^m\Big\|P_{(n_{j_{i-1}},n_{j_i}]}a_iw_{j_i}^{(i)}\Big\|^p
+ 4mp\delta^{1/p}+ D^p\sum_{i=1}^m\Big\|P_{[1,n_0]}a_iw_{j_i}^{(i)}\Big\|^p
+ \delta^{1/p}\\ &\leq
D^p\sum_{i=1}^m\Big\|P_{(n_0,\infty)}a_iw_{j_i}^{(i)}\Big\|^p +
D^p\sum_{i=1}^m\Big\|P_{[1,n_0]}a_iw_{j_i}^{(i)}\Big\|^p +
(4mp+1)\delta^{1/p}\\ & = D^p\sum_{i=1}^m\Big\|a_iw_{j_i}^{(i)}\Big\|^p +
(4mp+1)\delta^{1/p} = D^p\sum_{i=1}^m|a_i|^p + (4mp+1)\delta^{1/p}.
\end{align*}
A very similar calculation using \eqref{E:3.13.5} and
\eqref{E:3.13.7} yields
\begin{equation*}
\Big\|\sum_{i=1}^ma_iw_{j_i}^{(i)}\Big\|^p
\geq\frac{1}{C^p}\sum_{i=1}^m|a_i|^p - (4mp+1)\delta^{1/p}. \end{equation*}
As $\delta$ can be chosen arbitrarily close to zero we deduce the desired
conclusion.
\end{proof}

\subsection{Asymptotic structure}\label{subsec:2.4}
In this last preliminary subsection we recall the notion of asymptotic structure and its connection to weakly null trees. For $k\in\N$ we denote by $\cE_k$ the set of all norms on $\R^k$, for which the unit vector basis $(e_i)_{i=1}^k$ is a normalized monotone basis. With an easily understood abuse of terminology this can also be referred to as the set of all pairs $(E,(e_j)_{j=1}^k)$, where $E$ is a $k$-dimensional Banach space and $(e_j)_{j=1}^k$ is a normalized monotone basis of $E$.

We define a metric $\delta_k$ on $\cE_k$  as follows :
 For two spaces $E=(\R^k,\|\cdot\|_E)$ and $F=(\R^k,\|\cdot\|_F)$  we let $\delta_k(E,F)=\log\big( \|I_{E,F}\|\cdot \|I_{E,F}^{-1}\|\big)$, where
$I_{E,F}:E\to F$,  is the formal identity. It is also well known and easy to show that $(\cE_k, \delta_k)$ is a compact metric space.

We let $[\N]^{<\omega}=\{ S\ksubset \N: |S| < \infty \}$ and $[\N]^{\omega}=\{ S\ksubset \N: |S|= \infty\}$.
For $k\in \N$  we put $[\N]^{\le k}=\{ S\ksubset \N: |S|\kleq k\}$, and $[\N]^n=\{ S\ksubset \N: |S| = n\}$,
 and we always list the elements of some $\mb\in [\N]^{\kleq k}$ in increasing order, \ie if we write $\mb=\{m_1,m_2,\ldots, m_k\}$, we tacitly assume that $m_1<m_2<\ldots<m_k $.
If $X$ is a Banach space we call a tree $(x_{\nb}:\nb\in[\N]^{\le k})$ in $X$ {\em normalized} if $x_{\nb}\in S_X$, for all $\nb\in[\N]^{\le k}$,  and weakly convergent, or weakly null if for all $0\le j\le k-1$ and $n_1<n_2<\dots<n_j$, we have that $(x_{(n_1,n_2,\dots,n_j,i)})_i$ is weakly converging or  weakly null, respectively.

The following definition is due to Maurey, Milman, and Tomczak-Jaegermann  \cite{MaureyMilmanTomczak1995}. Here $S_X$ denotes the unit sphere in $X$, while $B_X$ denotes the closed unit ball.
\begin{defin}{(The  $k$-th asymptotic structure of $X$ \cite{MaureyMilmanTomczak1995}.)}\label{D:3.1}

Let $X$ be a Banach space. We denote by $\cof(X)$ the set of all its closed finite codimensional subspaces of $X$. For $k\in\N$ we define the {\em $k$-th asymptotic structure of $X$} to be the set, denoted by $\{X\}_k$, of spaces $E=(\R^k,\|\cdot\|)\in \cE_k$ for which the following is true:
 \begin{align}\label{E:3.1.1}
 \forall \vp\kgr0\, \forall X_1\kin\cof(X)\,&\exists x_1\kin S_{X_1}\, \forall X_2\kin\cof(X)\,\exists x_2\kin S_{X_2}\,\ldots  \forall X_k\kin\cof(X)\,\exists x_k\kin S_{X_k}\,\\
 &(x_j)_{j=1}^k\sim_{1+\vp} (e_j)_{j=1}^k.\notag
 \end{align}
 For $1\le p\le \infty$ and $c\ge 1$, we say  that $X$ {\em is $c$-asymptotically $\ell_p$}, if for  all  $k\kin \N$ and all spaces $E\in\{X\}_k$, with monotone normalized basis  $(e_j)_{j=1}^k$, $(e_j)_{j=1}^k$ is $c$-equivalent to the $\ell_p^k$ unit vector basis.
 We say  that $X$ {\em is asymptotically $\ell_p$}, if it is $c$-asymptotically $\ell_p$ for some $c\ge 1$. In case that $p=\infty$ we say that the space $X$ is $c$-asymptotically $c_0$, or asymptotically $c_0$.
\end{defin}

We denote by $T^*$ the Banach space constructed by Tsirelson in
 \cite{Tsirelson1974}. This is the archetype of a reflexive asymptotic-$c_0$
 space (see Remark \ref{R:6.2}).
 Soon after, in \cite{FigielJohnson1974}, it became clear that the
 easier  to define space is $T$, the dual of $T^*$, because the norm of this space is more conveniently described. It has since become common to refer to $T$ as Tsirelson space instead of $T^*$. Figiel and Johnson in \cite{FigielJohnson1974} gave an implicit formula that describes the norm of $T$ as follows. We call a sequence $(E_j)_{j=1}^n$ of finite subsets of $\mathbb{N}$ {\em admissible } if $n\le E_1<E_2<\cdots<E_n$.  For $x=\sum_{j=1}^\infty \lambda_j e_j\in c_{00}$ and $E\in [\N]^{<\omega}$ we write
 $E(x)=\sum_{j\in E} \lambda_j e_j$. As it was observed in \cite{FigielJohnson1974}, if  $\|\cdot\|_T$ denotes the norm of $T$ then for every $x\in c_{00}$:
   \begin{equation}\label{E:6.1}
    \|x\|_T=\max\Big\{\|x\|_\infty, \frac12 \sup\sum_{j=1}^n \|E_j(x)\|_T\Big\},
   \end{equation}
where the supremum is taken over all $n\in\N$ and admissible sequences $(E_j)_{j=1}^n$. The space $T$ is the completion of $c_{00}$ with this norm and the unit vector basis is a 1-unconditional basis of $T$.

It is worth noting that a $T^*$-sum of infinitely many infinite dimensional Banach spaces cannot be asymptotic-$c_0$.

\begin{lem}
\label{lem:2.8}
The space $(\oplus_{k=1}^\infty X_k)_{T^*}$ cannot be asymptotic-$c_0$ if infinitely many of the $X_k$'s are infinite dimensional.
\end{lem}

\begin{proof}
Let $L = \{k_1<k_2<\cdots\}$ denote the collection of $k\in\mathbb{N}$ for which $X_k$ is infinite dimensional. If any one of these $X_k$'s contains an isomorphic copy of $\ell_1$ we are done. Otherwise, by Rosenthal's theorem, we can pick for each $k\in L$ a normalized weakly null sequence $(x_i^{(k)})_i$ in $X_k$. For each $n\in\mathbb{N}$ take the countably branching weakly null tree $\{x_{\mb}:\mb\in[\mathbb{N}]^{\leq n}\}$ where $x_{\{m_1,\ldots,m_i\}} = x^{(k_i)}_{m_i}$. Every maximal branch of this tree is isometrically equivalent to  elements of $(e_{k_i})_{i=1}^n$, where $(e_i)_{i=1}^\infty$ denotes the unit  vector basis of $T^*$. Then $(e_{k_i})_{i=1}^n\in\{ (\oplus_{k=1}^\infty X_k)_{T^*}\}_n$ for all $n\in\mathbb{N}$. But $(e_{k_i})_{i=1}^\infty$ is not equivalent to the $c_0$ unit vector basis.
\end{proof}

 The following lemma, which will be used repeatedly follows from \cite[Proposition 2.3]{OdellSchlumprecht2002}, says in particular that, for a separable reflexive space every $N$-dimensional asymptotic subspace can be realized (up to an arbitrarily small perturbation) on a branch of a normalized weakly null tree of height $N$.

\begin{lem}
\label{lem:2.7}
Let $X$ be a Banach space with a separable dual, $k\in\N$, $(e_i)_{i=1}^k\in\{X\}_k$, and let $\vp>0$. Then there exists a countably branching weakly null tree $\{x_{\nb}:\nb\in[\N]^{\leq k}\setminus \{\emptyset\}\}$ in $S_X$, all of whose branches are $(1+\vp)$-equivalent to $(e_i)_{i=1}^k$.
\end{lem}

\section{Co-analyticity of $\mathsf{R} \cap \mathsf{As}_{c_0}$ and Hamming-type metrics}
\label{sec:3}
In Section \ref{sec:3.1} we expand on the general principles, mentioned in the Introduction, that are useful to estimate the projective complexity of classes of Banach spaces using certain bi-Lipschitz invariants. We show how such a strategy can be applied to show the co-analyticity of  the class of all separable and reflexive asymptotic-$\co$ Banach spaces using Theorem \ref{T:A}. In Section \ref{sec:3.2} we prove Theorem \ref{T:A}.

\subsection{Co-analyticity via bi-Lipschitz embeddings}
\label{sec:3.1}
The goal of this subsection is to  prove Corollary \ref{C:B}.  We will deduce it from the following Theorem
 which presents an, at least formal,  strengthening of Theorem \ref{T:A} and
 which will be proved in Subsection \ref{sec:3.2}.

\begin{thm}[Theorem \ref{T:A}]\label{T:3.2}\
\begin{enumerate}
\item Let $X$ be a separable reflexive Banach space. Then $X$ is  asymptotic-$c_0$ if and only if for all $1$-suppression unconditional
sequence $\bar e = (e_j)_j$ such that
$\lim_k\diam\big([\N]^k,d_{\eb}^{(k)}\big)=\infty$ one has $\sup_{k\in\N}c_X\big([\N]^k,d^{(k)}_{\eb}\big)=\infty$.

\item Moreover, if $X$ is a separable but not asymptotic-$c_0$ Banach space, then there is a  $1$-suppression unconditional
sequence $\bar e = (e_j)_j$, with
$\lim_k\diam\big([\N]^k,d_{\eb}^{(k)}\big)=\infty$
and for every $k\in\N$ a bi-Lipschitz embedding of $\big([\N]^k,d^{(k)}_{\eb}\big)$ of distortion at most $3$.
\end{enumerate}
\end{thm}

Before to deduce Corollary \ref{C:B}, note that the class of all separable and reflexive asymptotic-$\co$ Banach spaces is not analytic. For, if it were analytic, then
by \cite[Theorem 3]{DodosFerenczi2008} there would exist a separable reflexive Banach space that would contain
isomorphic copies of all separable and reflexive asymptotic-$c_0$ Banach spaces. But it was observed in \cite[Remark on Page 120]{OdellSchlumprechtZsak2008} that such a space cannot exist, and thus invoking Corollary \ref{C:B} and Souslin's theorem (see, e.g., \cite[Theorem 14.11]{Kechris1995}) which stipulates that a set is analytic and co-analytic if and only if it is Borel, we have:

\begin{cor}\label{T:3.1}
The class of all separable and reflexive asymptotic-$\co$ Banach spaces is co-analytic non-Borel in the Effros-Borel structure of closed subspaces of $C[0,1]$.
\end{cor}

We first fix some notation and make some remarks.
 Denote by $\mathsf{SB}$ the collection of all closed subspaces of the separable Banach space $C[0,1]$, endowed with the
Effros-Borel structure. This is a collection of Borel sets generated by a canonical Polish topology. This structure is very useful to ``measure" the complexity of classes of Banach spaces. We refer the reader
to the fundamental work of B. Bossard on this subject \cite{Bossard2002}.
Consider for a metric space $(M,d)$ and $D\ge 1$ the class \begin{equation*}
\mathsf{LC}^D_M:=\{Y\in\mathsf{SB}\mid M \text{ bi-Lipschitzly embeds into Y with distortion at most } D\}.
\end{equation*}

It is folklore (and not difficult but rather tedious to show) that the class $\mathsf{LC}^D_M$ is analytic, \ie the continuous image of a Polish space. So if we were to prove that a certain class of Banach spaces $\mathcal{B}$ coincides with a class of the form $\mathsf{LC}^D_M$ for some metric space $M$ then we could conclude that $\mathcal{B}$ is analytic. As a concrete example consider the class $\mathsf{SR}$ of all separable super-reflexive Banach spaces. It is known \cite{Baudier2007} that $\mathsf{SR}^c=\mathsf{LC}^D_{\bin_\infty}$ where $\bin_\infty$ is the binary tree of infinite height and $D\ge 1$ is a universal constant, and thus $\mathsf{SR}$ is co-analytic. Bourgain's original metric characterization of super-reflexivity \cite{Bourgain1986a} (from which \cite{Baudier2007} builds on) is in terms of the sequence of binary trees $(\bin_k)_{k\in \N}$, and could be reformulated as: there exists $D\ge 1$ such that
\begin{equation}\label{Eq:1}
\mathsf{SR}^c=\bigcap_{k\in\N}\mathsf{LC}^D_{\bin_k}.
\end{equation}
Since the  countable intersection of analytic sets is analytic, this gives another proof of the co-analyticity of $\mathsf{SR}$. Similarly, we could immediately deduce Corollary  \ref{C:B} if in Theorem \ref{T:A} we could replace all sequences of metric spaces of the form $([\N]^k,d^{(k)}_{\bar e})$ by a countable sub-collection.
But this is not possible as will be shown in Proposition  \ref{P:3.10}.
We overcome the problem of the uncountability   by representing the class of reflexive an asymptotic Banach spaces by a Souslin Scheme.
We  consider the following four classes of Banach spaces:
\begin{equation*}
\mathsf{R} = \{Y\in\mathsf{SB}:Y \text{ is reflexive}\},
\end{equation*}
\begin{equation*}
\mathsf{As}_{c_0} = \{Y\in\mathsf{SB}:Y \text{ is asymptotic-}c_0\},
\end{equation*}
\begin{equation*}
\mathsf{SU} = \{\bar e: \bar e = (e_i)_i \subset C[0,1]\text{ is a normalized $1$-suppression unconditional basic sequence}\},
\end{equation*}
\begin{equation*}
\mathsf{HU} = \{\bar e: \bar e\in \mathsf{SU} \text{ with } \lim_{k\in\N}\diam\big([\N]^k,d_{\eb}^{(k)}\big)=\infty\}.
\end{equation*}

Keeping in mind descriptive set theoretic applications, Theorem \ref{T:3.2} can be succinctly reformulated as
\begin{equation}
\mathsf{R}\cap\mathsf{As}_{\co}^c=\mathsf{R}\cap\bigcup_{\bar e\in\mathsf{HU}}\bigcap_{k\in\N}\mathsf{LC}^3_{\big([\N]^k,d^{(k)}_{\eb}\big)}.
\end{equation}

\begin{proof}[Proof of Corollary \ref{C:B}, using Theorem \ref{T:3.2}] We need  to show that   $\mathsf{R}\cap \mathsf{As}_{\co}$ is co-analytic.

A compactness argument implies that there exists a countable collection
 $\bar e^{(m)} = (\bar e^{(m)}_i)$, $m\in\mathbb{N}$, so that for every $\bar e\in \mathsf{SU}$ and $k\in\mathbb{N}$
there is $m\in\mathbb{N}$ so that $d_{\bar e}^{(k)}$ and $d_{\bar e^{(m)}}^{(k)}$ are $4/3$ equivalent.
Indeed, for fixed $k$, choose a countable set $\big((E^{(k)}_j, (e^{(k,j)}_i)_{i=1}^k):j\in\N\big)$ of $k$-dimensional subspaces with an 1-supression unconditional  and normalized basis
which is dense in the set of all  $k$-dimensional subspaces with an 1-supression unconditional  and normalized basis,  with respect to the metric introduced
at the beginning  of Subsection \ref{subsec:2.4}. For every $k,j\in\N$ choose an arbitrary
extension of $(e^{(k,j)}_i)_{i=1}^k$ into an infinite 1-supression unconditional and normalized basic sequence  $\bar e^{(k,j)}=(e^{(k,j)}_i)_{i=1}^\infty$.
Finally  reorder $\big(\bar e^{(k,j)}\big)_{k,j}$ into $\big(\bar e^{(m)}\big)_m$.

For simplicity denote $M_m^{(k)} := ([\mathbb{N}]^k,d_{\bar e^{(m)}}^{(k)})$, for $m,k\in\mathbb{N}$. Let
$$\mathcal{T} = \big\{((m_i,k_i))_{i=1}^n: n\in\mathbb{N}, \diam(M_{m_i}^{(k_j)}) \geq j,\text{ for all }1\leq j\leq  i\leq n\big\},$$
and observe that $\mathcal{T}$ is a countable, infinitely branching tree of infinite height (as partial order we just consider the  extension of finite sequences of pairs of natural numbers). Denote by
$$[\mathcal{T}] = \{ ((m_i,k_i))_{i=1}^\infty:  ((m_i,k_i))_{i=1}^n\in\mathcal{T}\text{ for all $n\in\mathbb{N}$}\},$$
the collection of branches of $\mathcal{T}$.
For $m,k\in\mathbb{N}$, define
$$\mathsf{LC}_{(m,k)} = \big\{Y\in\mathsf{SB}:  M_m^{(k)} \text{ embeds bi-Lipschitzly into $Y$ with distortion at most $4$}\}.$$
Recall that $\mathsf{LC}_{(m,k)}$ is an analytic set. A crucial observation is that the set
$\mathcal{M} := \cup_{\sigma\in[\mathcal{T}]}\cap_{n=1}^\infty\mathsf{LC}_{\sigma(n)}$ is also analytic since it is obtained via a Suslin operation of analytic sets. The properties of distances $d_{\bar e^{(m)}}^{(k)}$ and the  second part of
 Theorem \ref{T:3.2} imply that
\begin{equation}\label{Eq:3}
(\mathsf{As}_{c_0})^c\cap\mathsf{R}\subset\mathcal{M}.
\end{equation}
Additionally, the first part  in Theorem \ref{T:3.2} yields that $\mathsf{R}\cap\mathsf{As}_{c_0}\cap\mathcal{M} = \emptyset$ or equivalently
\begin{equation}\label{Eq:4}
\mathcal{M}\subset (\mathsf{R} \cap \mathsf{As}_{c_0})^c= (\mathsf{R})^c\cup(\mathsf{As}_{c_0})^c.
\end{equation}
Indeed, if a Banach space $X$ belongs to $\mathcal{M}$, then there exists an infinite branch $((m_i,k_i))_{i=1}^\infty$ in $[\mathcal{T}]$ such that $M_{m_i}^{(k_i)}$ embeds bi-Lipschitzly into $X$ with distortion at most $4$. Then a compactness argument yields the existence of $\bar e \in \mathsf{SU}$ and a sequence $(l_j)_j$ such that for all $i\in \N$, $(e^{(m_{l_j})}_1,...,e^{(m_{l_j})}_{k_i})_{j\ge i}$ tends to $(e_1,...,e_{k_i})$ for the Banach-Mazur distance. It then follows from our construction of $\mathcal{T}$ that $\bar e \in \mathsf{HU}$ and that for all $i\in \N$, $([\N]^{k_i},d_{\bar e}^{(k_i)})$ embeds bi-Lipschitzly into $X$ with distortion at most $4$. Since the sequence $(k_i)_i$ cannot be bounded, we deduce from the first part  of Theorem \ref{T:3.1} that $X$ is not in $\mathsf{R}\cap\mathsf{As}_{c_0}$.

It now follows from \eqref{Eq:3}, \eqref{Eq:4}, and elementary set-theoretic manipulations, that
\begin{equation}
(\mathsf{R} \cap \mathsf{As}_{c_0})^c=(\mathsf{R})^c\cup \mathcal{M}.
\end{equation}
We already observed that $\mathcal{M}$ is analytic and it is known (see  \cite[Corollary 3.3]{Bossard1997}) that the set $(\mathsf{R})^c$ is analytic.
Analyticity being preserved by taking finite unions, it follows that $\mathsf{R} \cap \mathsf{As}_{c_0}$ is co-analytic.

\end{proof}

\subsection{A bi-Lipschitz characterization of asymptotic-$\co$ spaces in the reflexive setting}
\label{sec:3.2}
In this section we pay our debt to Section \ref{sec:3.1} and prove Theorem \ref{T:3.2} (and thus Theorem \ref{T:A}). We will prove the two implications separately. But first we gather some essential properties of those metrics that are naturally generated by $1$-suppression unconditional sequences, and which play a central role in this section.   We call a basic sequence $(e_i)$
{\em $c$-suppression unconditional,} for some $c\ge 1$, if
for any $(a_i)\subset c_{00}$ and any $A\subset \N$
$$\Big\|\sum_{i\in A} a_i e_i\Big\|\le c \Big\|\sum_{i=1}^\infty a_i e_i\Big\|.$$
We call $(e_i)$ $c$-unconditional if for any $(a_i)\subset c_{00}$ and any $(\sigma_i)\in\{\pm 1\}^\N$
$$\Big\|\sum_{i=1}^\infty   a_i e_i\Big\|\le c \Big\|\sum_{i=1}^\infty  \sigma_i a_i e_i\Big\|.$$
Note that a $c$-unconditional basic sequence is $c$-suppression unconditional, and that any  $c$-suppression unconditional is $2c$-unconditional.

Recall from the introduction that for an arbitrary normalized $1$-suppression unconditional basis $\eb = (e_j)_{j\in\N}$ of a Banach space  $(E,\|\ \|)$, we define for every $k\kin\N$ a map $d^{(k)}_{\eb}:[\N]^k\times[\N]^k\to[0,\infty)$ such that for every $\mb=\{m_1,m_2,\ldots,m_k\}$ and $\nb=\{n_1,n_2,\ldots,n_k\}$ in $[\N]^k$
\begin{equation}
d^{(k)}_{\eb}(A,B) = \big\|\sum_{j\in F}e_j\big\|, \text{ where }F = \{j: m_j\neq n_j\}.
\end{equation}

The only metric axiom which is not trivially satisfied and that needs attention to ensure that the map $d^{(k)}_{\eb}$ is a genuine metric is the triangle inequality. This is where the unconditionality condition is needed. If $\mb=\{m_1,\ldots,m_k\}$, $\nb=\{n_1,\ldots,n_k\}$, and $\lb = \{l_1,\ldots,l_k\}$, set $F = \{j: m_j\neq n_j\}$, $G = \{j: m_j\neq l_j\}$, and $H = \{j: n_j\neq l_j\}$.
Since the set $F\subset G\cup H$  we have
$$F = F\cap(G\cup H) = (F\cap G)\cup ((F\setminus G)\cap H).$$
It follows from $1$-suppression unconditionality that
\begin{align*}
d_{\eb}^{(k)}(\mb,\nb) &= \Big\|\sum_{j\in F}e_j\Big\| \leq  \Big\|\sum_{j\in F\cap G}e_j\Big\| +  \Big\|\sum_{j\in (F\setminus G)\cap H}e_j\Big\| \\
&\leq  \Big\|\sum_{j\in G}e_j\Big\|+  \Big\|\sum_{j\in H}e_j\Big\|
= d_{\eb}^{(k)}(\mb,\lb) + d_{\eb}^{(k)}(\lb,\nb).
\end{align*}

The metric $d^{(k)}_{\eb}$ is similar to the Hamming metric in the sense that for $\mb = \{m_1,\ldots,m_k\}$ and  $\nb = \{n_1,\ldots,n_k\}$ the distance $d_{\eb}^{(k)}(\mb,\nb)$ is determined by the set $F\subset \{1,2,\ldots,k\}$ of coordinates $i$ on which $m_i$ and $n_i$ differ. The following important features directly follow from the definition of the metric and classical Banach space theory.
\begin{lem}\label{L:3.3}
Let $\eb = (e_j)_{j\in \N}$ be a normalized 1-suppression unconditional basis of a Banach space  $(E,\|\ \|)$.
\begin{itemize}
\item[(i)] If $\eb = (e_j)_{j\in\N}$ is the unit vector basis of $\ell_1$ then $d^{(k)}_{\eb}$ is the Hamming distance $d^{(k)}_\ham$ on $[\N]^k$. Hence, for any normalized 1-suppression unconditional basic sequence $\eb = (e_j)_{j\in\N}$ and any $\mb$, $\nb$ in $[\N]^k$ we have $d^{(k)}_{\eb}(\mb,\nb) \leq d^{(k)}_\ham(\mb,\nb)$.
\item[(ii)]For every $k\in\N$ and every $\M\in[\N]^\omega$ we have $$\diam([\M]^k,d_{\eb}^{(k)})= \Big\|\sum_{j=1}^{k}e_j\Big\|.$$
In particular, $\lim_k\mathrm{diam}\big([\N]^k,d_{\eb}^{(k)}\big) = \infty$ if and only if $\eb = (e_j)_{j\in\N}$ is not equivalent to the unit vector basis of $c_0$.


\end{itemize}
\end{lem}

The domination of the metric $d^{(k)}_{\eb}$ by the Hamming metric allows us to use the concentration inequality from \cite{BLMS_JIMJ20} to prove the non-embeddability implication of Theorem \ref{T:3.2}. Indeed, assume that $Y$ is asymptotic-$c_0$ and reflexive, and let $\bar e = (e_j)_j$ be a normalized $1$-suppression unconditional sequence such that
$\lim_k\diam\big([\M]^k,d_{\eb}^{(k)}\big)=\infty$.
The crucial observation here is that the domination property in Lemma \ref{L:3.3} (i), can be equivalently restated by saying that the identity maps from  $([\N]^k, d_{\ham})$ to $([\N]^k, d_{\eb})$ are $1$-Lipschitz, and a straightforward application of \cite[Theorem B]{BLMS_JIMJ20} shows that there exists $C\in [1,\infty)$ so that for every $\eb\in \mathsf{SU}$, every $k\in\N$ and, every $1$-Lipschitz map $f:\big([\N]^k,d^{(k)}_{\eb}\big)\to Y$ there exists $\M\in[\N]^\omega$ so that
\begin{equation}\label{Eq:7}
\diam\big( f([\M]^k)\big) \leq C.
\end{equation}

If moreover $\eb\in \mathsf{HU}$, inequality \eqref{Eq:7} and (ii) of Lemma \ref{L:3.3} clearly prevent the equi-bi-Lipschitz embeddability of the sequence $\big([\N]^k,d^{(k)}_{\eb}\big)_{k}$, or in other words $\sup_{k\in\N}c_Y\big([\N]^k,d^{(k)}_{\eb}\big)=\infty$ necessarily. We thus proved:

\begin{thm}\label{T:5.4} Let $X$ be a separable asymptotic-$c_0$ reflexive Banach space. Then for all $1$-suppression unconditional
sequence $\bar e = (e_j)_j$ such that
$\lim_k\inf_{\M\in[\N]^\omega}\diam\big([\M]^k,d_{\eb}^{(k)}\big)=\infty$ one has $$\sup_{k\in\N}c_X\big([\N]^k,d^{(k)}_{\eb}\big)=\infty.$$
\end{thm}

For the remaining implication  and the ``moreover" part  of Theorem \ref{T:3.2} we may assume that $X$ does not contain an isomorphic copy o $\ell_1$. Indeed, it is clear that the graphs $\ham_k^\omega$, $k\in \N$, embed isometrically into $\ell_1$. For $X$ separable, but not containing $\ell_1$, we will  use  the following  result by  Freeman, Odell, Sari, and Zheng.

\begin{thm}\label{T:3.5} \cite{FOSZ2017}*{Theorem 4.6}
If a separable Banach space $X$ does not contain any isomorphic copy of $\ell_1$ and all the asymptotic models generated by normalized weakly null arrays are
equivalent to the $c_0$ unit vector basis, then $X$ is asymptotically $c_0$.
\end{thm}

Theorem \ref{T:3.5} establishes a crucial connection between asymptotic models and asymptotic structure in the extremal $\co$-case. In the light of the new information of Theorem \ref{T:3.5}, the completion of the proof of Theorem \ref{T:3.2} boils down to showing that if a separable reflexive Banach that admits at least one asymptotic model generated by normalized weakly null arrays that is not equivalent to the $c_0$ unit vector basis, contains equi-bi-Lipschitzly a sequence $\big([\N]^k,d^{(k)}_{\eb}\big)_{k}$, for some $\eb\in \mathsf{HU}$.
Slightly anticipating the ensuing argument, Lemma \ref{L:3.3} (ii) says that if $(e_i)$ is an asymptotic model (generated by a normalized weakly null array) that is not equivalent to the $c_0$ unit vector basis, then $(e_i)\in \mathsf{HU}$. This observation provides a natural candidate for the embedding map. Indeed, arrays (and in turn asymptotic models) are intimately connected to Hamming-type metrics in the sense that if $\big(x^{(i)}_j:i, j\in \N\big)\subset S_X$ is an infinite array, then the map $\phi:[\N]^k\to X$ defined for any $\mb = \{m_1,m_2,\ldots,m_k\}$ by
\[\phi(\mb) = \sum_{i=1}^kx^{(i)}_{m_i}.\]
 is clearly $1$-Lipschitz with respect to $d_{\ham}$. As we will shortly see if the array generates a $1$-suppression unconditional asymptotic model $\eb$ we can slightly modify $\phi$ by ``pushing the vectors far enough along the sequence'' and obtain a map that is Lipschitz (with a slightly larger distortion) with respect to $d_{\eb}$. Estimating the lower Lipschitz bound however will require a strengthening of the unconditionality condition, and is the content of the crucial  Lemma \ref{L:3.10} below. This is done via the notion of joint spreading models recently introduced by Argyros, Georgiou, Lagos, and Motakis \cite{AGLM2017}, a notion that we briefly recall together with some ingredients needed in the proof of Lemma \ref{L:3.10}.

 \begin{defin}[Plegmas]\label{D:3.8}\cite{AKT2013}*{Definition 3}
Let $k,m\in\N$ and $s_i = (s^{(i)}_1, s^{(i)}_2,\ldots,s^{(i)}_m)\subset \N$ for $i=1,\ldots, k$. The family $(s_i)_{i=1}^k$ is called a {\em plegma } if
$$s^{(1)}_1 < s_1^{(2)} < \cdots< s_1^{(k)} < s_2^{(1)}< s_2^{(2)}<\cdots< s_2^{(k)}<\cdots<s_m^{(1)} < s^{(2)}_m < \cdots < s^{(k)}_m.$$
\end{defin}

A family $\big(x^{(i)}_j:i=1,2,\ldots, k, j\in \N\big)\subset X,$ will be referred to an {\em array of height $k$ in $X$}, and we can extend naturally the terminology for infinite arrays introduced in Section \ref{sec:3.2} to arrays of finite height.

\begin{defin}[Joint spreading models]\label{D:3.9} \cite{AGLM2017}*{Definition 3.1}
Let   $\big(x_j^{(i)}: 1\kleq i\kleq k, j\kin\N\big) $ and $\big(e_j^{(i)}: 1\kleq i\kleq  k, j\kin\N\big)$ be two normalized arrays of height $k$ in the Banach spaces $X$, and $E$, respectively, whose rows are normalized and basic. We say that  $(x_j^{(i)}:1\kleq i\kleq  k,  j\kin\N) $  {\em generates  $(e_j^{(i)}: 1\kleq i\kleq  k, j\kin\N)$  as a joint spreading model} if there exists a null sequence of positive real numbers $(\vp_m)_{m=1}^\infty$ so that for every $m\in\N$, every  plegma $(s_i )_{i=1}^k$, $s_i=(s^{(i)}_j :j=1,2,\ldots, m)$ for $1\leq i\leq k$, with $\min(s_1)=s^{(1)}_1 \geq m$, and scalars $((a_j^{(i)})_{j=1}^m)_{i=1}^k$ in $[-1,1]$ we have
\[\Bigg|\Big\|\sum_{j=1}^m\sum_{i=1}^k a_j^{(i)}x^{(i)}_{s^{(i)}_j}\Big\|_X - \Big\|\sum_{j=1}^m\sum_{i=1}^ka_j^{(i)}e^{(i)}_{j}\Big\|_E\Bigg|<\vp_m.\]
\end{defin}

Joint spreading models are naturally related to spreading models  as well as asymptotic models. If  $\big(x_j^{(i)}: 1\kleq i\kleq  k, j\kin\N\big) $ generates  $\big(e_j^{(i)}:  1\kleq i\kleq  k, j\kin\N\big)$ as a joint spreading model, then  $(e_j^{(i)})_{j=1}^\infty$   is the spreading model of $(x^{(i)}_j)_{j=1}^\infty$, for $i=1,2,\ldots, k$. On the other hand, if  $k\in\N$ and  $\big(x^{(i)}_j:i=1,2,\ldots, k, j\in\N\big)\subset S_X$   if  a normalized  weakly null array of height $k$, then we extend this array to an infinite array $\big(x^{(i)}_j:i=1,2,\ldots, k, j\in \N\big)$, by letting $$x^{(sk+i)}_j=x^{(i)}_j, \text{ for $s\in\N$ and $i=1,2, \ldots, k$.}$$ By Proposition \ref{P:3.5}  we can pass to a subarray  $(z^{(i)}_{j}: i\in\N, j\in\N)$  of  $(x^{(i)}_{j}: i\in\N, j\in\N)$ which admits an asymptotic model $(e_j)$. Now letting $e^{(i)}_{j}= e_{(j-1)k+ i}$, for $i=1,2, \ldots, k$ and $j\in\N$ we observe that
the array $(e_j^{(i)}:1\kleq i\kleq  k,  j\kin\N)$ is the {\em joint spreading model of  $(z^{(i)}_{j}: i=1,2, \ldots, k, j\in\N)$}.
In particular this argument shows that joint spreading models of normalized weakly null arrays are $1$-supression unconditional.
\begin{lem}\label{L:3.10}
Let $X$ be a Banach space and $(x_j^{(i)}:1\kleq i\kleq k, j\kin\N)$ be a  normalized weakly null array of height  $k$. Then for every $\vp>0$ and $m\in\N$ there exists $\L\in[\N]^\omega$ so that for every $i_1,\ldots,i_m$ in $\{1,\ldots,k\}$ (not necessarily different) and pairwise different $l_1,\ldots,l_m\in \L$ the sequence $(x_{l_j}^{(i_j)})_{j=1}^m$ is $(1+\vp)$-suppression unconditional.
\end{lem}

\begin{proof}
As explained above, we may assume after passing to a subarray that  $(x_j^{(i)}:1\kleq i\kleq k, j\kin\N)$ generates a joint spreading model  $(e_j^{(i)}:1\kleq i\kleq k, j\kin\N)$ that is 1-suppression unconditional. Thus, we find $N\in\N$, so that for any plegma $(s_i)_{i=1}^k$, $s_i=(s^{(i)}_1, s^{(i)}_2,\ldots, s^{(i)}_m)$, for $i=1,2,\ldots, k$, with $N\leq s^{(1)}_1$ the family $\big(x^{(i)}_{s^{(i)}_j}: 1\kleq i\kleq k, 1\kleq j\kleq m\big)$ is $(1+\vp)$-suppression unconditional. Let $\L$ be the set that consists of all positive integers multiple of $2k$ that are greater than $N+k$.

Let now $i_1,\ldots,i_m$ in $\{1,\ldots,k\}$ and $l_1,\ldots,l_m$ be pairwise different elements of $\L$.
 After reordering, we can assume $l_1\kle l_2\kle \ldots\kle l_m$. Let $r_1\kle r_2\kle \ldots\kle r_m$ be in $\N$ so that $l_j=2k r_j$. We will now define a plegma $(s_i)_{i=1}^k$,
 $s_i=(s^{(i)}_{j})_{j=1}^m$, as follows. First
   we define $s^{(i_j)}_j\keq l_j\keq2k r_j$, for $j\keq1,2,\ldots, m$. Then,
   since $l_{j+1}-l_{j}\kge 2k$, for every $j\keq1,\ldots m-1$ and $s_1^{(i_1)}> N+k$, we can find  natural numbers $s^{(i_j)}_j\kle s^{(i_j+1)}_j\kle s^{(i_j+2)}_j\kle\ldots s^{(k)}_j\kle s^{(1)}_{j+1}\kle\ldots\kle s^{(i_{j+1})}_{j+1}$, numbers
    $N \kle s^{(1)}_1\kle s^{(2)}_1\kle \ldots \kle s^{(i_1-1)}_1\kle s^{(i_1)}_1$ and numbers $s^{(i_m)}_m\kle s^{(i_m+1)}_{m}\kle  \ldots \kle s^{(k)}_m$, which means that
    the family $(s_i)_{i=1}^k$, with  $s_i=(s_j^{(i)})_{j=1}^m$, for $i=1,2,\ldots,k$ is a plegma. Thus $\big(x^{(i)}_{s_j^{(i)}}:i\keq1,2,\ldots,k,j\keq1,2,\ldots m\big)$ is $(1\kplus\vp)$-suppression unconditional
  and $(x^{(i_j)}_{l_j})_{j=1}^m$ is just a subsequence of it.
  \end{proof}

Having now established all the tools we needed we can proceed with the proof of:
 \begin{thm}\label{T:3.11}
Let $X$ be a Banach space and $\eb = (e_j)_{j\in\N}$ be an asymptotic model generated by a normalized weakly null array in $X$. Then, for any $k\in\N$ and $\vp>0$, the metric space $([\N]^k,d^{(k)}_{\eb})$ bi-Lipschitzly embeds into $X$ with distortion at most $(2+\vp)$.
\end{thm}

\begin{proof}
Let $\big(x^{(i)}_j):i,j\kin\N\big)$ be a normalized weakly null array in $X$ that generates an asymptotic model $\eb=(e_j)_{j\in\N}$. Fixing $k\in\N$ and $\delta>0$ and passing to appropriate subsequences of the array we may assume that for any $j_1<\cdots<j_k$ and any $a_1,\ldots,a_k$ in $[-1,1]$ we have
\begin{equation}
\label{E:4.4.1}
\Bigg|\Big\|\sum_{i=1}^ka_ix^{(i)}_{j_i}\Big\| - \Big\|\sum_{i=1}^ka_ie_i\Big\|\Bigg| < \delta.
\end{equation}
In addition,  by applying Lemma \ref{L:3.10} we may also assume that for any $i_1,\ldots,i_{2k}$ in $\{1,\ldots,k\}$ and any pairwise different $l_1,\ldots,l_{2k}$ in $\N$ the sequence $(x^{(i_j)}_{l_j})_{j=1}^{2k}$ is $(1+\delta)$-suppression unconditional.

We are now ready to define the embedding. Define $\phi:[\N]^k\to X$ as follows. If $\mb = \{m_1,m_2,\ldots,m_k\}$ set
\[\phi(\mb) = \sum_{i=1}^kx^{(i)}_{km_i+i}.\]
Observe first that for $m_1 <\cdots<m_k$ we have $km_1 + 1< km_2+2<\cdots<km_k+k$. Then, if $\mb = \{m_1,\ldots,m_k\}$,
$\nb = \{n_1,\ldots,n_k\}$ and $F = \{i:m_i\neq n_i\}$ we have
$$\phi(\mb) - \phi(\nb) = \sum_{i\in F}x^{(i)}_{km_i+i} - \sum_{i\in F}x^{(i)}_{kn_i+i}.$$
It immediately follows from the triangle inequality and \eqref{E:4.4.1} that if $\mb\neq \nb$ then
$$\|\phi(\mb) - \phi(\nb)\| \leq 2\|\sum_{i\in F}e_i\| + 2\delta \leq 2(1+\delta)d_{\eb}^{(k)}(\mb,\nb).$$

Also, note that $km_i + i = kn_{i'} + i'$ if and only if $i = i'$ and $m_i = n_{i'}$. We deduce that the sequence $(x^{(i)}_{km_i+i})_{i\in F}\cup(x^{(i)}_{kn_i+i})_{i\in F}$ is $(1+\delta)$-suppression unconditional. Therefore we have
\[
\|\phi(\mb) - \phi(\nb)\| \geq \frac{1}{(1+\delta)}\Big\|\sum_{i\in F} x^{(i)}_{km_i+i}\Big\| \geq \frac{1}{(1+\delta)}\Big(\Big\|\sum_{i\in F}e_i\Big\| - \delta\Big) \geq \frac{(1-\delta)}{(1+\delta)}d_{\eb}^{(k)}(\mb,\nb).
\]
Hence, the distortion of $\phi$ is at most $2(1+\delta)^2/(1-\delta)$. For a given $\vp>0$, we choose $\delta>0$ small enough,  and then deduce the result.
\end{proof}

As we observed earlier Theorem \ref{T:3.11} implies the remaining implication of Theorem \ref{T:3.1} as well as the ``moreover'' part  via Theorem \ref{T:3.5}.

\medskip
At the end of this section we would like to address  the question  whether or not  in the class of reflexive spaces, the property of not  being asymptotic $c_0$
could be characterized  by the  uniform   Lipschitz embedability of $([\N]^k, \bar e)$, $k\in\N$, for some $\bar e$, where $\bar e$ only comes out of a countable subset of $ \mathsf{HU}$. This is not the case as the following Proposition shows.

\begin{prop} \label{P:3.10}
Let
$$D\subset \left\{ (d^{(k)})_{k\in\N}: \begin{matrix}d^{(k)} \text{ is a metric on $[\N]^k$, which is dominated by $d_H^{(k)}$ and }\\ \limsup_{k\to\infty}\inf_{\M\in[\N]^\omega} \diam([\M]^k,d^{(k)}) =\infty\end{matrix}\right\}$$
 be countable.
Then there exists a  reflexive Banach space $X$,  which is not asymptotic $c_0$, so that for all $ (d^{(k)})_{k\in\N}\in D$ and
for all sequences $(\Psi_k)$, where $\Psi_k: ([\N]^k, d^{(k)} )\to X$ is $1$-Lipschitz it follows that
$$\lim_{k\to\infty} \inf_{\M\in[\N]^{\omega}} \frac{\diam (\Psi_k([\M]^k, d^{(k)} ))}{\diam ([\M]^k, d^{(k)})}=0,$$
in  particular the $\Psi_k$ cannot be uniform bi-Lipschitz embeddings.
\end{prop}
\begin{proof}
Let $D=\big \{(d^{(k)}_n:k\in\N): n\in\N\big\}$ and for $n\in\N$, put
$f_n(k)=\inf_{\M\in [\N]^{<\omega}}\diam([\M]^k,d_n^{(k)}).$
For each $n$ there exists a $k_n$ so that
$$\min_{m\le n} f_m(k)\ge n\text{ for all $k\ge k_n$.}$$
We put  $\tilde f(k)=1 $ if $k<k_1$, and
   $\tilde f(k)= \min_{m\le n} f_m(k)\ge n $ whenever $k_n\le k<k_{n+1}$.
Then put
$$f(k)=\max\big(2,\min\big({\tilde f^{1/2}(k)}, \log_2(1+k)\big)\big).$$
It follows that
\begin{equation}\label{E:1.6}
\lim_{k\to\infty} f(k)=\infty,\quad
\lim_{k\to\infty} \frac{f(k)}{f_n(k)}=0 \text{ and }\lim_{k\to\infty} \frac{f(k)}{k^{1/n}}=0, \text{ for all $n\in\N$}.
\end{equation}
The space $X$ will be the dual of the space $Z=Z_f$, which was constructed in \cite{Schlumprecht1991}.
Although $f$ does not satisfy all the conditions demanded in the construction there, for our purposes the properties in
 \eqref{E:1.6} suffice. By \cite[Proposition 2]{Schlumprecht1991}, there is a Banach space  $Z$ with a $1$-subsymmetric basis $(e_i)$, whose norm satisfies
 the following implicit equation:
 \begin{align}\label{E:1.7}
 \|x\|=&\max\big(\|x\|_\infty,\sup_{2\le l\le \infty} \|x\|_l\big), \text{ where }\\
 & \|x\|_l=\frac1{f(l)}  \max_{E_1<E_2<\ldots E_l} \sum_{j=1}^l\|E_j(x)\|, \text{ for $l\ge 2$, and $x\in X$}.\notag
 \end{align}
It is clear that, by \eqref{E:1.6} and  \eqref{E:1.7},  $Z$ does not contain $c_0$.
 We will show that $Z$ does also not contain a copy of
  $\ell_1$. This fact follows from the arguments in \cite{Schlumprecht1991} (more precisely the arguments on  page 87), but for the sake of better readability
  let us give a self contained  proof.
 Assume $Z$ contained a normalized block sequence $(x_n)$ which is equivalent to the $\ell_1$ unit basis.
By James's Theorem \cite{James1964} we can  assume it  is  $(1+\vp)$-equivalent to the $\ell_1$ unit basis, for some given $\vp>0$.
 It follows for any $l\in \N$ and any $A\subset \N$,  with $|A|\geq l/\vp$,  that there are finite sets $E_1<E_2<\ldots E_l$,
  so that (letting $m_1=1$, and $m_j=\max\{n: \supp(x_{j-1})\cap E_n\not=\emptyset\}$, if $1<j\le l+1$)
\begin{equation}\label{E:1.8}
\Big\|\frac1{|A|} \sum_{j\in A} x_j\Big\|_l =\frac1{ f(l)} \frac1{|A|}\sum_{j=1}^l \|E_j(x)\| \le
\frac1{ f(l)} \sum_{j=1}^l \Big\|\sum_{i=m_j}^{m_{j+1}} x_i\Big\|
   \le \frac1{f(l)} \frac{|A|+l}{|A|}  \le \vp +\frac1{f(l)}.\end{equation}
Secondly we choose a {\em rapidly increasing sequence of $\ell_1$-averages of length $2$}, a name coined by Gowers and Maurey
\cite{gowersmaurey1997}. By this we mean that we first choose $l_1\in \N$ so that $1/f(l)<\vp$, for all $l\ge l_1$, then we choose $n_1\ge  l_1/\vp$ and
$z_1=\frac1{n_1}\sum_{j=1}^{n_1} x_j$. Then we choose $l_2\in \N$ so that  $\max\supp(z_1) <\vp f(l)$, for $l\ge l_2$,
 $n_2\ge l_2/\vp$  and then $z_2=\frac{1} {n_2}\sum_{j=n_1+1}^{n_1+n_2} x_j$.

 It follows from \eqref{E:1.8}   for some $l\ge 2$ that
 $$\|z_1+z_2\|=\|z_1+z_2\|_l\le \|z_1\|_l+\|z_2\|_l  \le   \begin{cases}    2\vp+  2\frac1{f(l)}\le 2\vp + 1  &\text{ if $2\le l\le l_1$,}\\
                                                                                                           1+ \vp +\frac1{f(l)} \le 1+2\vp &\text{ if $l_1 <l\le l_2$,}\\
                                                                                                           \vp +1 &\text{ if $l_2 <l$.}\\
                                                                                                          \end{cases}$$
  But this contradicts the assumption that $(x_j)$ is $(1+\vp)$-equivalent to the unit vector basis of $\ell_1$ if $\vp>0$ is chosen small enough.
Since $(e_j)$ is an unconditional basis it follows from the fact that $Z$ does neither contain $c_0$ nor $\ell_1$, that
  $Z$ is reflexive \cite{James1950}.  Since $(e_j)$ is subsymetric $Z$ cannot even  be asymptotic $\ell_1$.
   It follows that $(e^*_n)$ is a $1$-subsymmetric  basis of $Z^*$  and
   by a straightforward dualization argument  \cite{MaureyMilmanTomczak1995}*{Theorem 4.3} $Z^*$ is not asymptotically $c_0$.
  From \eqref{E:1.7}  it follows that  for any normalized  block
basis $(x^*_j)_{j=1}^n$ in $B_{Z^*}$ we have for an appropriate $x\in S_Z$, and letting $E_j=\supp(x^*_j)$ for $j=1,2\ldots, n$
\begin{equation}\label{E:1.3}
\Big\|\sum_{j=1}^n x^*_j\Big\|=\sum_{j=1}^n x^*_j(x)\le \sum_{j=1}^n \|E_j( x)\|\le f(n).
 \end{equation}
Assume now that $n,k\in\N$ and that $\Psi:([\N]^k, d_n^{(k)}) \to Z^*$ is $1$-Lipschitz, and let $\vp>0$.
By  \cite[Proposition 4.1]{BaudierLancienSchlumprecht2018}  there is an $\M'\in[\M]^\omega$ and a $y\in Z^*$ and for all $\bar m\in [\M']^k$ there is a block sequence  $(y^{(j)}_{\bar m})_{j=1}^k\subset B_{Z^*}$ so that
$$\Big\|\psi(\bar m)-y-\sum_{j=1}^k y^{(j)}_{\bar m} \Big\|\le \vp \text{ for all $\bar m\in [\M']^k$.}$$
Thus
$$\|\Psi(\bar m)-\Psi(\bar n)\|\leq 2\vp+ \| y^{(1)}_{\bar m}+  y^{(2)}_{\bar m} + \ldots y^{(k)}_{\bar m} -y^{(1)}_{\bar n}+  y^{(2)}_{\bar n} + \ldots y^{(k)}_{\bar n}\|\le 2\vp + 2f(k).$$
which by the second property in \eqref{E:1.6} proves our claim.
\end{proof}

\section{Embeddability of Hamming graphs into non asymptotic-$\co$ spaces}
\label{sec:4}
In this section we discuss coarse embeddability of the Hamming graphs into non asymptotic-$\co$ spaces. Notably, we show that $T^*(T^*)$ is reflexive non-asymptotic-$\co$ space in which the Hamming graphs cannot be coarsely embedded in certain canonical ways.

\subsection{Embeddability into $\big(\oplus_{n=1}^\infty\ell_p^n(T^*))_{T^*}$}
\label{sec:4.1}
For $p\in[1,\infty]$, the space $\big(\oplus_{n=1}^\infty\ell_p^n(T^*))_{T^*}$ is separable and reflexive  but not asymptotically-$c_0$, yet all its spreading models are uniformly equivalent to the unit vector basis of $c_0$. More precisely, we have.

\begin{prop}\label{P:7.4} Let $p\in[1,\infty]$.
Every spreading model generated by a normalized weakly null sequence in $(\oplus_{n=1}^\infty\ell_p^n(T^*))_{T^*}$ is 6-equivalent to the unit vector basis of $c_0$.
\end{prop}

\begin{proof}
Every  normalized block basis $(x_n)$  in $V = (\oplus_{n=1}^\infty\ell_p^n)_{T^*}$ has  a subsequence which isometrically equivalent to a  $(x_n)$ in $T^*$ and thus  has a spreading model equivalent to the $c_0$-unit basis with lower bound $1$ and upper bound $2$, and  therefore for any finitely supported vector $x_0$
and any $k$ there are $n_1<n_2\cdots<n_k$ so that $\{x_0\}\cup \{x_{n_j}, j=1,2,\ldots k\}$  is equivalent to the $\ell_\infty^{n+1}$ basis, with lower
bound $1$ and upper bound $3$. Since $V(T^*)$ is canonically isometric to $\big(\oplus_{n=1}^\infty\ell_p^n(T^*))_{T^*}$, our claim follows from Proposition \ref{P:3.12}.
\end{proof}

It turns out that despite all its spreading models generated by a normalized weakly null sequence are 6-equivalent to the unit vector basis of $c_0$, the space $(\oplus_{k=1}^\infty\ell_p^k(T^*))_{T^*}$ contains equi-coarsely the Hamming graphs.

\begin{prop}\label{P:7.3}
Let $1\leq p<\infty$. The Hamming graphs embed equi-coarsely into the Banach space $(\oplus_{k=1}^\infty\ell_p^k(T^*))_{T^*}$.
\end{prop}

\begin{proof}
Consider for every $n\in\mathbb{N}$ the space $\ell_p^k(T^*)$ and let $(e_j^{(i)})_j$ denote the standard basis of the $i$'th copy of $T^*$. Then, for any $j_1<\cdots<j_k$ the sequence $(e^{(i)}_{j_i})_{i=1}^k$ is isometrically equivalent to the unit vector basis of $\ell_p^k$. Additionally, the collection $(e_j^{(i)}: j\in\mathbb{N}, 1\leq i\leq k\}$ is 1-unconditional. We conclude that if we define the map $f_k:[\mathbb{N}]^k\to\ell_p^k(T^*)$ with $f_k({\mb}) = \sum_{i=1}^k e^{(i)}_{m_i}$, where ${\mb} = \{m_1,\ldots,m_k\}$, then for all ${\mb}, \nb\in[\mathbb{N}]^k$ we have
\[d^{(k)}_\ham(\mb,\nb)^{1/p} \leq \|f_k(\mb) - f_k(\nb)\|  \leq 2 d^{(k)}_\ham(\mb,\nb)^{1/p}.\]
We now deduce that the Hamming graphs equi-coarsely embed into the space $(\oplus_{k=1}^\infty\ell_p^k(T^*))_{T^*}$ with compression modulus $\rho(t) = t^{1/p}$ and expansion modulus $\omega(t) = 2t^{1/p}$.
\end{proof}

The proof actually gives that the $\frac{1}{p}$-snowflaking of the $k$-dimensional Hamming graph, \ie $([\N]^k,d_\ham^{1/p})$, bi-Lipschitzly embeds into $\ell_p^k(T^*)$ with distortion at most $2$. In particular, the Hamming graphs equi-bi-Lipschitzly embed into $(\oplus_{k=1}^\infty\ell_1^k(T^*))_{T^*}$.

\begin{remark}\label{R:7.5} For $k\in \N$, the Johnson graph of height $k$ is the set $[\N]^k$ equipped with the  metric defined by $d_{\mathsf{J}}^{(k)}(\mb,\nb)=\frac12 \sharp(\mb \triangle \nb)$ for $\mb,\nb \in [\N]^k$. It is proved in \cite{BaudierLancienSchlumprecht2018} that there is a constant $C\ge 1$ such that for any $k\in \N$ and $f \colon ([\N]^k,d_{\mathsf{J}}^{(k)}) \to T^*$ Lipschitz, there exists $\M \in [\N]^\omega$ so that $\diam (f([\M]^k))\le C\Lip (f)$. It is easily seen that the same is true if $T^*$ is replaced by any reflexive asymptotic-$c_0$ space. However, we do not know whether the Johnson graphs embed equi-coarsely into $(\oplus_{n=1}^\infty\ell_p^n(T^*))_{T^*}$. The reason is that canonical embeddings of the Johnson graphs are built on sequences and not arrays. This confirms the qualitative difference between asymptotic models and spreading models. The space $\big(\oplus_{n=1}^\infty\ell_p^n(T^*)\big)_{T^*}$ is a possible example of a space that equi-coarsely contains the Hamming graphs but not the Johnson graphs.
\end{remark}

\begin{prob}\label{P:7.5}
Does there exist a Banach space equi-coarsely containing the Hamming graphs and not the Johnson graphs? Is $(\oplus_{n=1}^\infty\ell_p^n(T^*))_{T^*}$ such an example?
\end{prob}

\subsection{Embeddability into $T^*(T^*)$}
\label{sec:4.2}
We now introduce and study a relaxation of the asymptotic-$\co$ property that is relevant to the coarse geometry of the Hamming graphs.
\subsubsection{A partial obstruction: the asymptotic-subsequential-$\co$ property}
We denote the unit vector basis of $T^*$ by $(e^*_j) $, which is also $1$-unconditional. Therefore the space $T^*(T^*) = (\oplus_{k=1}^\infty T^*)_{T^*}$ is well defined. We study the asymptotic properties of this space and the goal is to prove that the space $T^*(T^*)$, which is not an asymptotic-$c_0$ space by Lemma \ref{lem:2.8}, is very close to being one. We introduce the following definition.

\begin{defin}
Let $X$ be an infinite dimensional Banach space and $1\leq p\leq\infty$. We say that $X$ is an {\em asymptotic-subsequential-$\ell_p$} space if there exists a constant $C\geq 1$ so that for all $n\in\mathbb{N}$ there exists an $N\in\mathbb{N}$ satisfying the following: whenever $(e_i)_{i=1}^N$ is in $\{X\}_N$ (recall Definition \ref{D:3.1}) then there are $i_1<\cdots<i_n$ so that $(e_{i_k})_{k=1}^n$ is $C$-equivalent to the unit vector basis of $\ell_p^n$.
\end{defin}
Clearly, any asymptotic-$\ell_p$ space fits the above description. To follow our previously introduced  convention, we shall use the term asymptotic-subsequential-$c_0$ space for the case $p=\infty$. We do not know whether such spaces fail to contain the Hamming graphs equi-coarsely, nonetheless this property rules out certain ``canonical'' embeddings as described below

\begin{prop}\label{P:6.9} If $Y$ is an asymptotic-subsequential-$c_0$ space then there is no sequence of maps $(f_k)_k$, such that $f_k: \ham^\omega_k\to Y$, and where
$(f_k)_k$ is a sequence of equi-coarse embeddings of $({\ham}^\omega_k)_k$ into  $Y$ with the property that for every $k\in \N$ there is a normalized weakly null array $(y^{(i)}_j: 1\le i\le k, j\kin\N)$ so that
$$f_k(\mb)=\sum_{i=1}^k y^{(i)}_{m_i} ,\text{ for all $\mb=\{m_1,m_2,\ldots ,m_k\}\in[\N]^k$.}$$
\end{prop}

\begin{proof}
Let $Y$ be a $C$-asymptotic-subsequential-$c_0$ space and let us fix an increasing sequence of non-negative real numbers $(\rho_n)_n$. Let us assume that for every $k\in\N$ we can find a normalized weakly null array $(y^{(i)}_j: 1\le i\le k, j\kin\N)$ in $Y$ so that for all $m\leq k$, all $i_1<\cdots<i_m$ and $j_1<\cdots<j_m$ we have $\|\sum_{l=1}^my^{(i_l)}_{j_l}\| \geq \rho_m$. We pass to a subarray that generates a finite asymptotic model $(e_i)_{i=1}^k$. This asymptotic model has the property that for all $1\leq m\leq k$ and $1\leq i_1<\cdots<i_m\leq n$ we have $\|\sum_{l=1}^me_{i_l}\| \geq \rho_m$. Additionally, $(e_i)_{i=1}^k\in\{X\}_k$. Since this is the case for all $m,k\in\N$ we can easily conclude using the definition of $C$-asymptotic-subsequential-$c_0$ that $\rho_m\leq C$ for all $m\in\N$. But this means that  $(f_k)_k$, defined above, is not a sequence of equi-coarse embeddings of
$({\ham}^\omega_k)_{k\in \N}$ into $Y$.
\end{proof}

\begin{rem}\label{R:6.10}
The above proof with minor modifications shows that a reflexive asymptotic-subsequential-$c_0$ space $Y$ cannot have the following property:
\begin{enumerate}
\item[(\dag)]
There are sequences $\big(\rho(n)\big)_n,\big(\mu(n)\big)_n\subset (0,\infty)$ with $\rho(n),\mu(n)\nearrow\infty$, if $n\nearrow\infty$, and for each $k\in\N$ a weakly
null tree $(y^{(k)}_{\nb})_{\nb\in[\N]^{\le k}}\subset B_Y$, so that for all
$k\in\N$ and all $\mb,\nb\in [\N]^k$ , $\mb=\{m_1,m_2,\ldots,m_k\}$, and  $\nb=\{n_1,n_2,\ldots,n_k\}$
$$\rho\big(d^{(k)}_\ham(\mb,\nb)\big)\le\Big\| \sum_{i=1, m_i\not=n_i}^k  y^{(k)}_{\{m_1,m_2,\ldots, m_i\}}- y^{(k)}_{\{n_1,n_2,\ldots, n_i\}}\Big\|$$
and $$\rho\big(d^{(k)}_\ham(\mb,\nb)\big)\le\Big\| \sum_{i=1}^k  y^{(k)}_{\{m_1,m_2,\ldots, m_i\}}- y^{(k)}_{\{n_1,n_2,\ldots, n_i\}}\Big\|\le\mu(d^{(k)}_\ham(\mb,\nb)).$$
\end{enumerate}
The existence of trees $(y_{\mb}^{(k)}: \mb\in [\N]^k)$ satisfying the condition $(\dag)$ above, means that the maps
$$f_k: \ham^{\omega}_k\to Y,\quad \{m_1,m_2,\ldots, m_k\}\mapsto \sum_{i=0}^k y^{(k)}_{\{m_1,\ldots, m_i\}}$$
are equi-coarse embeddings, and that the lower bound for $\|f_k(\mb)-f_k(\nb)\|$ is witnessed by the values of $y^{(k)}_{\{m_1,m_2,\ldots, m_i\}}- y^{(k)}_{\{n_1,n_2,\ldots, n_i\}}$, where $m_i\not=n_i$,
for $\mb=\{m_1,m_2,\ldots,m_k\}$, and  $\nb=\{n_1,n_2,\ldots,n_k\}$ in $[\N]^k$.
\end{rem}

\subsubsection{$T^*(T^*)$ is asymptotic-subsequential-$\co$}\ \\

The main goal of this section is to prove that $T^*(T^*) $ is  asymptotic-subsequential-$c_0$ and thereby finishing the proof of   Theorem \ref{T:C}. We start with some preparatory work. The following property of $T^*$ (see \cite{Tsirelson1974}*{Lemma 4}) is essential:
\begin{equation}\label{E:6.1a}
\Big\| \sum_{j=1}^n x_j\Big\|_{T^*}\le 2\max_{1\leq j\leq n}\|x_j\|_{T^*} \text{ whenever $(x_j)_{j=1}^n$ is a block sequence, with $n\le \supp(x_1)$.}
\end{equation}
and thus, under a slightly weaker condition
\begin{equation}\label{E:6.2}
\Big\| \sum_{j=1}^n x_j\Big\|_{T^*}\le 3\max_{1\leq j\leq n}\|x_j\|_{T^*} \text{ whenever $(x_j)_{j=1}^n$ is a block sequence, with $n\le \supp(x_2)$.}
\end{equation}
\begin{rem}
\label{R:6.2}
The fact that $T^*$ is $2$-asymptotic-$c_0$ is an easy consequence of the above estimate \eqref{E:6.1a}. This well known fact
 is hard to track down in the literature, and follows from the fact that every weakly null tree admits an refinement for which all branches are  arbitrary small perturbations of blocks.
  A noteworthy comment is that in \cite{OdellSchlumprechtZsak2008} the notion of asymptotic-$\ell_p$, $1\leq p\leq\infty$ with respect to a finite dimensional decomposition (FDD) was introduced and it was proved that a reflexive space is asymptotic-$\ell_p$ if and only if it linearly embeds in a space that is asymptotic-$\ell_p$ with respect to an FDD.
 \end{rem}

Recall that the norm of $T$ satisfies the implicit formula \eqref{E:6.1}. We will need the following observation for the space $T^*$, which follows from a statement for $T$, proved in
\cite{CasazzaOdell1983}*{Theorem 2}.

\begin{prop}
\label{P:6.3}
There exists a constant $D_M>0$ so that the following holds. For every $n\in\mathbb{N}$,  any vectors $x_1,\ldots,x_n$ in $T^*$, having disjoint supports, with $\min(\mathrm{supp}(x_k))\ge n$, for $1\leq k\leq n$, it follows that
$$\Big\|\sum_{k=1}^nx_k\Big\|_{T^*} \leq D_M \max_{1\leq k\leq n}\|x_k\|_{T^*}.$$
\end{prop}
Note that in Proposition \ref{P:6.3}, the vectors have disjoint supports (as opposed to consecutive supports as in \eqref{E:6.2}). In order to prove Proposition \ref{P:6.3} we need to introduce some necessary notions. A norm very similar to $\|\cdot\|_T$ was defined by W. B. Johnson in \cite{Johnson1976b}. It is called the {\em modified Tsirelson norm}, we denote this norm  by $\|\cdot\|_M$ and it satisfies the implicit formula
\begin{equation}\label{E:6.3}
\|x\|_M = \max\Big\{\|x\|_\infty,\frac{1}{2}\sup\sum_{k=1}^n\|E_k(x)\|_M\Big\}\end{equation}
where the supremum is taken over all $n\in\mathbb{N}$ and \emph{disjoint} subsets $(E_k)_{k=1}^n$ of $\mathbb{N}$ with $ n \leq \min(E_k)$ for $1\leq k\leq n$. Note that there is a unique norm $\|\cdot\|_M$ satisfying this implicit formula (this can, e.g., be shown by induction on the size of the support of the vector $x$). The main statement we need to prove Proposition \ref{P:6.3} is the following.
\begin{thm} {\rm (\cite{CasazzaOdell1983}*{Theorem 2}, see also \cite{CasazzaShura1989}*{Theorem V.3})}
\label{T:6.4}

There exists a constant $C_M>0$ so that for any sequence of scalars $(a_i)_{i=1}^n$ we have
\[\Big\|\sum_{i=1}^na_ie_i\Big\|_T\leq \Big\|\sum_{i=1}^na_ie_i\Big\|_M \leq C_M\Big\|\sum_{i=1}^na_ie_i\Big\|_T.\]
\end{thm}

\begin{proof}[Proof of Proposition \ref{P:6.3}]
Let $x_1,x_2,\ldots,x_n\in T^*$ have pairwise disjoint support with $\min(\supp(x_j))\ge n$, for $j=1,2,\ldots, n$. We first choose $y\in S_T$, with
$y(\sum_{j=1}^n x_j)=\big\| \sum_{j=1}^n x_j\big\|_{T^*}$. By the  1-unconditionality of the basis of $T$, we can assume that $\supp(y)\subset \bigcup_{j=1}^n \supp(x_j)$, and
letting $y_j=\supp(x_j)(y)$ we deduce from Theorem \ref{T:6.4}  and \eqref{E:6.3} that
\begin{align*}
\Big\|\sum_{j=1}^n x_j\Big\|_{T^*}&=\sum_{j=1}^n y_j(x_j)\le \sum_{j=1}^n \|y_j\|_T\cdot \max_{j=1,\ldots,n} \|x_j\|_{T^*}\le \sum_{j=1}^n \|y_j\|_M\cdot \max_{j=1,\ldots,n} \|x_j\|_{T^*}\\
&\le 2\Big\|\sum_{j=1}^n y_j\Big\|_M\cdot \max_{j=1,\ldots,n} \|x_j\|_{T^*}
\le 2C_M \Big\|\sum_{j=1}^n y_j\Big\|_T\cdot \max_{j=1,\ldots,n} \|x_j\|_{T^*}
\le 2C_M    \max_{j=1,\ldots,n} \|x_j\|_{T^*},
\end{align*}
which implies our claim if we choose $D_M=2C_M$.
\end{proof}

We denote the basis of $T^*$  now by $(e_j)$.
For $A\subset \N$ we denote by $P_A$ the projection
$$P_A: T^*(T^*)\to T^*(T^*) ,\quad (x_n)\mapsto  (x_n)_{n\in A},$$
Note that
$$\big\|P_A\big((x_n)\big)\big\|=\Big\|\sum_{j\in A} \|x_j\|e_j\Big\|_{T^*}.$$
We call for $i\in \N$ the space $P_i(T^*(T^*))=P_{\{i\}}(T^*(T^*))\equiv T^*$, {\em the $i$'th component of } $(T^*(T^*))$
and we denote by  $(e^{(i)}_j)_j$ the basis of the $i$-th component (which is of course isometrically equivalent to $(e_j)$).
For $R\subset \N^2$ we denote by $P_R$ the (norm $1$) projection
$$P_R: T^*(T^*)\to T^*(T^*)\,\quad \sum_i \sum_j a_{(i,j)} e^{(i)}_j \mapsto \sum_{(i,j)\in R} a_{(i,j)} e^{(i)}_j .$$

The first out of two key Lemmas towards showing Theorem \ref{T:C} is the following
\begin{lem}\label{L: 2} Let $k\in\N$ and $k=n_0<n_1<\ldots <n_k$. For $j=1,2,\ldots, k$ put
 $R_j=(k,n_j]\times[1,n_j]$ and  let  $z_j\in P_{R_j\setminus R_{j-1}}(T^*(T^*))$, with $\|z_j\|\leq 1$.
 Then it follows for $(a_j)_{j=1}^k\subset \R$ that
 \begin{equation}
 \label{E:2.1} \Big\|\sum_{j=1}^k a_j z_j\Big\|\le 3D_M\max_{j=1,2,\ldots, k}|a_j|
 \end{equation}
\end{lem}

\begin{proof}
 For $j=1,2, \ldots ,k$ we write $z_j$ as
\begin{align*}
            z_j=\sum_{i=k+1}^{n_{j-1}}\underbrace{\sum_{s=n_{j-1}+1}^{n_j}  z_j (i,s) e^{(i)}_s}_{u^{(i)}_j, \text{ for $k<i\le n_{j-1}$}}
              + \sum_{i=n_{j-1}+1}^{n_j}
           \underbrace{ \sum_{s=1}^{n_j} z_j (i,s) e^{(i)}_s}_{u^{(i)}_j, \text{ for $n_{j-1}<i\le n_{j}$}}.
           \end{align*}
Thus
\begin{align*}
\sum_{j=1}^k  a_jz_j=
\sum_{j=1}^k  a_j\Bigg[\sum_{i=k+1}^{n_{j-1}}\sum_{s=n_{j-1}+1}^{n_j} z_j (i,s) e^{(i)}_s+ \sum_{i=n_{j-1}+1}^{n_j}
             \sum_{s=1}^{n_j} z_j (i,s) e^{(i)}_s \Bigg]=\sum_{i=k+1}^{n_k}  y^{(i)},
\end{align*}
where for $i=k+1,\ldots n_k$, say $ n_{j-1}<i \le n_j$, for some $j=1,2,\ldots k$ we have
$$y^{(i)} =P_i\Big(\sum_{j=1}^k  a_jz_j \Big)=
\sum_{l=j+1}^k a_l\sum_{s=n_{l-1}+1}^{n_l} z_l(i,s) e^{(i)}_s+
a_j\sum_{s=1}^{n_j}z_j(i,s)e^{(i)}_s= \sum_{l=j}^k a_l u^{(i)}_l.$$

The following picture visualizes the above decompositions.

  \begin{center}
  \includegraphics[scale=1]{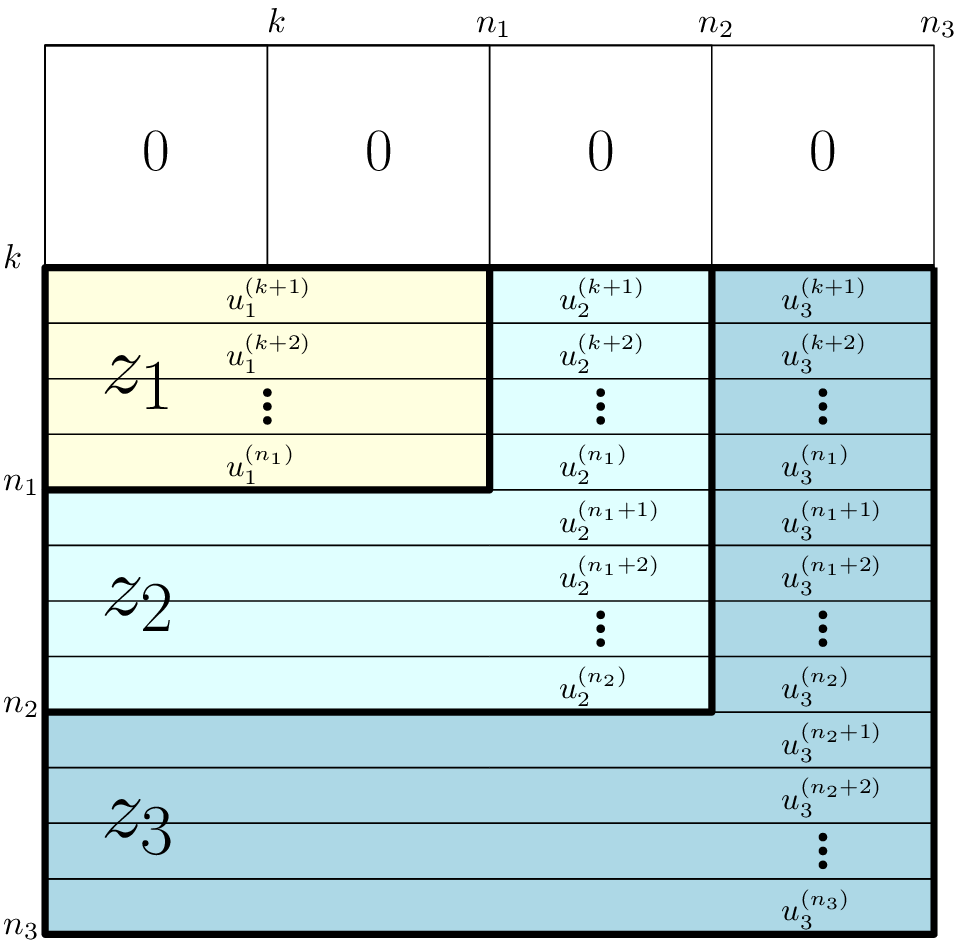}
  \end{center}

It follows from  \eqref{E:6.2} that for $n_{j-1}<i\le n_j$
\begin{equation}\label{E:3}
\big\| y^{(i)} \big\| \le 3\max_{l=j,\ldots  k} |a_l|\cdot\|u^{(i)}_{l}\|= 3|a_{l_i}|\cdot\|u^{(i)}_{l_i}\|
\end{equation}
where $j\le l_i\le k$ is a number for which above maximum is attained. For $j=1,2,\ldots k$ we
define $A_j=\{k< i\le n_k: l_i=j\}$. Then $(A_j)_{j=1}^k$ is a partition  of $\{k+1,\ldots,n_k\}$ and from Proposition \ref{P:6.3} and \eqref{E:3} we deduce that

\begin{align*}\Big\|\sum_{j=1}^k  a_jz_j\Big\|&= \Big\|\sum_{i=k+1}^{n_k} \big\| y^{(i)} \big\| e_i\Big\|_{T^*}\\
             & \le D_M \max_{j=1,\ldots,k} \Big\|\sum_{i\in A_j}\| y^{(i)} \big\| e_i\Big\|_{T^*}\\
             &\le 3D_M \max_{j=1,\ldots,k} \Big\|\sum_{i\in A_j} a_j \|u^{(i)}_j\| e_i\Big\|_{T^*}\\
             &\le 3 D_M\max_{j=1,\ldots,k} \Big\|\sum_{i=k+1}^{n_j}a_j \|u_j^{(i)}\| e_i\Big\|_{T^*}\\
             &= 3 D_M\max_{j=1,\ldots,k} |a_j|\|z_j\|\leq 3 D_M\max_{j=1,\ldots,k} |a_j|\|x_j\|\leq 3 D_M\max_{j=1,\ldots,k} |a_j|.
            \end{align*}
\end{proof}

The second key Lemma towards showing Theorem \ref{T:C} is the following
\begin{lem}\label{take care of the hat}
Let $k\in\N$, $M = k^{k+1}$, and $k=n_0<n_1<\ldots <n_M$. For $j=1,2,\ldots, k$ put
 $R_j=[1,k]\times[1,n_j]$ and  let  $w_j\in P_{R_j\setminus R_{j-1}}(T^*(T^*))$, with $\|w_j\|\leq 1$.
 Then, there exist $1\leq j_1<\cdots<j_k\leq M$ so that for $(a_\ell)_{\ell=1}^k\subset \R$ that
 \begin{equation}
 \label{E:2.1a} \Big\|\sum_{\ell=1}^k a_\ell w_{j_\ell}\Big\|\le 2\max_{\ell=1,2,\ldots, k}|a_\ell|
 \end{equation}
\end{lem}

\begin{proof}
Define $f:\{1,M\}\to[0,1]^k$ given by  $f(j) = (\|P_1w_j\|,\|P_2w_j\|,\ldots,\|P_kw_j\|)$. Next, write $[0,1] = \cup_{d=1}^kI_d$, where $I_1 = [0,1/k]$, $I_2 = (1/k,2/k]$,\ldots,$I_k = ((k-1)/k,1]$. Define
\[\mathcal{I} = \{I_{d_1}\times I_{d_2}\times\cdots\times I_{d_k}: (d_1,d_2,\ldots,d_k)\in\{1,\ldots,k\}^k\}.\]
Note that $\mathcal{I}$ forms a partition of $[0,1]^k$ into $k^k$ sets. By the pigeonhole principle and the fact that $ M/k^k = k$, there exist $1\leq j_1<\cdots<j_k\leq M$ and $(d^0_1,\ldots,d^0_k)\in \{1,\ldots,k\}^k$  so that for $1\leq \ell\leq k$, $f(j_\ell)\in I_{d^0_1}\times I_{d^0_2}\times\cdots\times I_{d^0_k}$. In particular, for $1\leq \ell\leq k$ and $1\leq i\leq k$ we have
\begin{equation}
\label{they are all the same I tell you}
\Big|\Big\|P_iw_{j_\ell}\Big\| - \Big\|P_iw_{j_1}\Big\|\Big| \leq \frac{1}{k},
\end{equation}
i.e., the value $\|P_iw_{j_\ell}\|$, up to error $1/k$, depends only on $i$ and not on $\ell$.

Finally, take $a_1,\ldots,a_k$ with $\max_{1\leq \ell\leq k}|a_\ell|  = 1$ and estimate
\begin{equation*}
\begin{split}
\Big\|\sum_{\ell=1}^k a_\ell w_{j_\ell}\Big\| &= \Big\|\sum_{i=1}^k\Big\|\sum_{\ell=1}^kP_i(a_\ell w_{j_\ell})\Big\|e_i\Big\| \stackrel{\eqref{E:6.1a}}{\leq} \Big\|\sum_{i=1}^k\max_{1\leq\ell\leq k}\Big(|a_\ell|\|P_i( w_{j_\ell})\|\Big)e_i\Big\|\\
&\stackrel{\eqref{they are all the same I tell you}}{\leq} \Big\|\sum_{i=1}^k\|P_i( w_{j_1})\|e_i\Big\| + \Big\|\sum_{i=1}^k\frac{1}{k}e_i\Big\| \leq \|w_{j_1}\| + 1\leq 2.
\end{split} 
\end{equation*}
\end{proof}

We combine the two Lemmas above to obtain the following, from which Theorem \ref{T:C} will follow.
\begin{prop}\label{blocks have subsequence c0}
Let $k\in\N$, $M = k^{k+1}$, and $k=n_0<n_1<\ldots <n_M$. For $j=1,2,\ldots, M$ put
 $R_j=[1,n_j]^2$ and  let  $x_j\in P_{R_j\setminus R_{j-1}}(T^*(T^*))$, with $\|x_j\|=1$. Then, there exist $1\leq j_1<\cdots<j_k\leq M$ so that $(x_{j_\ell})_{\ell=1}^k$ is $(3D_M + 2)$-equivalent to the unit vector basis of $\ell_\infty^k$.
 \end{prop}

\begin{proof}
 For $j=1,2, \ldots, M$ we write $x_j$ as
\begin{align*}
x_j &=\sum_{i=1}^{n_{j-1}}\sum_{s=n_{j-1}+1}^{n_j}  x_j (i,s) e^{(i)}_s+ \sum_{i=n_{j-1}+1}^{n_j}  \sum_{s=1}^{n_j} x_j (i,s) e^{(i)}_s
       =  w_j + z_j,\text{ where}\\
 w_j&=\sum_{i=1}^{k}\sum_{s=n_{j-1}+1}^{n_j}  x_j (i,s) e^{(i)}_s
            \text{ and }
            z_j=\sum_{i=k+1}^{n_{j-1}}\sum_{s=n_{j-1}+1}^{n_j}  x_j (i,s) e^{(i)}_s
              + \sum_{i=n_{j-1}+1}^{n_j}
           \sum_{s=1}^{n_j} x_j (i,s) e^{(i)}_s.
             \end{align*}
Then, $(w_j)_{j=1}^M$ satisfies the assumption of Lemma \ref{take care of the hat} and there exist $1\leq j_1<\cdots<j_k\leq M$ so that $(w_{j_\ell})_{\ell=1}^k$ is dominated by the unit vector basis of $\ell_\infty$ with constant 2. Finally, $(z_{j_\ell})_{\ell=1}^k$ satisfies the assumption of Lemma \ref{L: 2}, i.e., it is dominated by the unit vector basis of $\ell_\infty$ with constant $3D_M$.
\end{proof}

\begin{proof}[Proof of Theorem \ref{T:C}]
We already showed in Lemma \ref{lem:2.8} that $T^*(T^*)$ is not asymptotic $c_0$.
 Secondly, let $k\in\N$ let  $(f_j)_j^M$ be  the basis  of an element of the $M$-th asymptotic structure of $T^*(T^*)$, where $M = k^{k+1}$. Using  a straightforward perturbation argument, there is for any $\vp>0$
 a  block sequence $(x_j)_{j=1}^M$, satisfying the conditions of  Proposition \ref{blocks have subsequence c0},  for some sequence $k<n_1<n_2<\ldots< n_M$,
 which is $(1+\vp)$-equivalent to $(f_j)_{j=1}^M$. Thus, there is a subsequence $\big(f_{j_\ell}\big)_{\ell=1}^k$ that is $(1+\vp)(3DM+2)$-equivalent to the $\ell_\infty ^{k}$-unit basis.
\end{proof}

\section{Final remarks and open problems}\label{S:8}
Although we do not know whether or not the Hamming graphs equi-coarsely embed into $T^*(T^*)$ we now understand that if such embeddings were to exist they would not be of any of the canonical types that we have described in Proposition \ref{P:6.9} and Remark \ref{R:6.10}.

\begin{prob}\label{Prob:7.1}
Is it true that the Hamming graphs do not equi-coarsely embed into any reflexive asymptotic-subsequential-$c_0$ space? In particular, is it true that the Hamming graphs do not equi-coarsely embed into $T^*(T^*)$?
\end{prob}

The class of asymptotic-subsequential-$c_0$ spaces is a new one. This is not surprising as even proving that $T^*(T^*)$ has this property is non-trivial and the motivation for defining this property presented itself only now. A more general theorem can be shown, albeit with a more technical proof.

\begin{thm}
The $T^*$-sum of any sequence of $C$-asymptotic-$c_0$ spaces for a uniform constant $C$ is  asymptotic-subsequential-$c_0$. 
\end{thm}

Such examples contain many asymptotic-$c_0$ subspaces.
\begin{prob}\label{Prob:7.2}
Let $X$ be an infinite dimensional asymptotic-subsequential-$c_0$ space. Does $X$ contain an infinite dimensional asymptotic-$c_0$ subspace?
\end{prob}

Next we describe a particular Banach space and some of its properties which are interesting regarding the study of certain asymptotic properties under a metrical scope. This example is based on the original idea of Szlenk in \cite{Szlenk1968}. It is also related to \cite[Example 4.2]{OdellSchlumprecht2002}. For $1<p<\infty$ and $1\leq q\leq \infty$ we can construct a reflexive Banach space $X_p^{q,\omega}$ with the following property: all asymptotic models generated by normalized weakly null arrays in $X_p^{q,\omega}$ are isometrically equivalent to the unit vector basis of $\ell_p$, yet  $\ell_q^k$ is (isometrically) in the $k$-th asymptotic structure of  $X_p^{q,\omega}$ for every $k\in\N$. Therefore   a statement which is analogous  to Theorem \ref{T:3.5} for $\ell_p$, $1<p<\infty$, cannot be true.

The construction of the space $X_p^{q,\omega}$ that we are about to describe is based on the idea of Szlenk from \cite{Szlenk1968}, and is somewhat similar to \cite[Example 4.2]{OdellSchlumprecht2002}. Fix $1< p <\infty$ and $1\leq q\leq \infty$ and define by induction a sequence of spaces $(X^{q,k}_p)_k$ as follows. Set $X_p^{q,0} = \mathbb{R}$ and then set $X^{q,k}_p = \mathbb{R}\oplus_q\ell_p(X^{q,k-1}_p)$. Finally, define $X_p^{q,\omega} = (\oplus_{k=0}^\infty X_p^{q,k})_p$. Each space $X^{q,k}_p$ is reflexive and so is $X_p^{q,\omega}$. The fact that all asymptotic models generated by normalized weakly null arrays in $X_p^{q,\omega}$ are isometrically equivalent to the unit vector basis of $\ell_p$ can be proved as follows.  Use Proposition \ref{P:3.13} to show by induction that for all $k\in\N$ all the asymptotic models generated by normalized weakly null arrays in $X_p^{q,k}$ are isometrically equivalent to the $\ell_p$-unit vector basis, and use Proposition \ref{P:3.13} one more time to obtain the same conclusion for $X_p^{q,\omega}$. We now turn to the statement about the asymptotic structure of $X_p^{q,\omega}$.

\begin{prop}\label{P:7.6}
Let $p\in(1,\infty)$ and $q\in[1,\infty]$. For every $k\in\mathbb{N}\cup\{0\}$ the space $X^{q,k}_p$ contains a normalized weakly null tree $(x_{\mb}:\mb\in[\mathbb{N}]^{\leq k})$, all branches of which are isometrically equivalent to the unit vector basis of $\ell_q^k$.
\end{prop}

\begin{proof}
For $k=0$ pick a norm-one vector $x_\emptyset$ in $X^{q,0}_p = \mathbb{R}$. Let now $X^{q,k}_p = \mathbb{R}\oplus_q\ell_p(X^{q,k-1}_p)$ and let for each $i\in\mathbb{N}$ $(x^{(i)}_{\mb}: \mb\in[\mathbb{N}]^{\leq k-1})$ be a    normalized weakly null tree in the $i$'th copy of $X^{q,k-1}_p$ all branches of which are isometrically equivalent to the unit vector basis of $\ell_q^{k-1}$. Take $x_\emptyset$ to be a norm-one vector in $X^{q,k}_p$ that resides in $\mathbb{R}$ (the left part of the sum $X^{q,k}_p =\mathbb{R}\oplus_q\ell_p(X^{q,k-1}_p)$) and for $1\leq n\leq k$ and $\mb = \{m_1,\ldots,m_n\}$ define
$x_{\mb} = x^{m_1}_{\{m_2-m_1,\ldots,m_n-m_1\}}$ (in particular, for $\mb = \{m\}$, $x_{\mb} = x^m_\emptyset$).
\end{proof}

\begin{remark}\label{R:7.7}
For each $k\in\mathbb{N}\cup\{0\}$ the collection $(x_{\mb}: \mb\in[\mathbb{N}]^{\leq k})$ forms a 1-unconditional basis of $X^{q,k}_p$. Hence, the space $X_p^{q,\omega}$ has an unconditional basis.
\end{remark}

As previously mentioned, it follows from \cite[Lemma 3.5]{BLMS_JIMJ20} that every asymptotic space of $X_p^{q,\omega}$ is realized by a countably branching normalized weakly null tree and thus we obtain:

\begin{cor}\label{C:7.8}
Let $p\in(1,\infty)$ and $q\in[1,\infty]$. For every $k\in\mathbb{N}$ the unit vector basis of $\ell_q^k$ is in $\{X_p^{q,\omega}\}_k$.
\end{cor}

Recall the following notions of {\em asymptotic uniform convexity} and  {\em asymptotic uniform smoothness} that were introduced originally by Milman in \cite{Milman1971}, and with the following notation and terminology in \cite{JohnsonLindenstraussPreissSchechtman2002}.

\begin{defin}\label{D:7.9}
For a Banach space $X$ the {\em modulus of  asymptotic uniform smoothness\/} $\bar\rho_X(t)$ is given for $t>0$ by
$$\bar \rho_X(t) = \sup_{x\in S_X} \inf_{Y\in\cof(X)} \sup_{y\in S_Y}\|x+ty\| -1\ .$$

The modulus of asymptotic uniformly convexity $\bar\delta_X(t)$ is given for $t>0$ by
$$\bar\delta_X(t) = \inf_{x\in S_X} \sup_{Y\in\cof(X)}\inf_{y\in S_Y} \|x+ty\| -1\ .$$

$X$ is called  {\em asymptotically uniformly smooth} (AUS) if $\lim_{t\to 0^+} \bar \rho_X(t)/t = 0$, and
$X$ is called  {\em asymptotically uniformly convex\/} (AUC) if for $t>0$, $\bar \delta_X (t) >0$.
\end{defin}

Note that, as it was shown in \cite{BKL2010}, within the class of reflexive Banach spaces the subclass of reflexive spaces that admit an equivalent
asymptotic uniformly smooth norm (i.e., they are {\em AUS-able}) and admit an equivalent asymptotic uniformly convex norm (i.e., they are {\em AUC-able})
is coarse Lipschitzly rigid. It was later proved in \cite{BCDKRSZ2017} that, within the class of reflexive spaces with  an unconditional
asymptotic structure, the subclass of such spaces that are additionally AUC-able is coarse Lipschitzly rigid.
Whithin this context we are also inclined to study the metric properties of AUS-able spaces.
It is known that whenever a Banach space $X$ coarse Lipschitzly embeds into a reflexive AUS-able space $Y$ then $X$ is reflexive \cite{BKL2010}*{Theorem 4.1}.
We recall the important Problem $2$ from \cite{GLZ2014}.
\begin{prob}\label{Prob:7.10}
Is the class of reflexive AUS-able spaces coarse Lipschitzly rigid?
\end{prob}

We observe that an approach using asymptotic models to characterize reflexive AUS-able spaces in terms of equi-coarse-Lipschitz embeddability of the Hamming graphs, or similar metric spaces, is not easily possible. In particular, the space $X_2^{1,\omega}$ is a reflexive non-AUS-able space with an unconditional basis with only isometric $\ell_2$ asymptotic models. In other words, the information gained from knowing all the asymptotic models of this space cannot be used to reveal that the space is non-AUS-able.

\begin{cor}
\label{C:7.11}
Let $p\in(1,\infty)$. The space $X_p^{1,\omega}$ is non-AUS-able.
\end{cor}

\begin{proof}
By \cite[Theorem 3]{OdellSchlumprecht2006} if a Banach space with separable dual is AUS-able then there exists a $1<p<\infty$ so that all of its asymptotic spaces are uniformly dominated by the unit vector basis of $\ell_p$. Since by Corollary \ref{C:7.8},
 $\ell_1^k$ is in $\{X_p^{1,\omega}\}_k$, this space cannot be AUS-able.
\end{proof}

\end{document}